\documentclass[11pt]{preprint}
\usepackage{difftrees} 
\usepackage[full]{textcomp}
\usepackage[osf]{newtxtext} 
\usepackage[cal=boondoxo]{mathalfa}
\usepackage{colortbl}

\usepackage{tikz-cd} 
\usetikzlibrary{cd} 
\usepackage[all,cmtip]{xy}
\usepackage{comment}

\usepackage{amssymb}
\usepackage{mathtools}
\usepackage{hyperref}
\usepackage{breakurl}
\usepackage{mhenvs}
\usepackage{mhequ} 
\usepackage{mhsymb}
\usepackage{booktabs}
\usepackage{tikz}
\usetikzlibrary{external}
\tikzsetexternalprefix{tikzpics/}
\usepackage{tcolorbox}
\usepackage{mathrsfs}
\usepackage[utf8]{inputenc}
\usepackage{longtable}
\usepackage{wrapfig}
\usepackage{rotating} 
\usepackage{subcaption}
\usepackage{mathrsfs}
\usepackage{epsfig}
\usepackage{microtype}
\usepackage{comment}
\usepackage{wasysym}
\usepackage{centernot}
\usetikzlibrary{snakes}
\usepackage{enumitem}
\usepackage{bm}
\usepackage{stackrel}
\usepackage{graphicx}
\usepackage{axodraw}
\usepackage{xspace}
\usepackage{subcaption} 
\usepackage{epsfig} 
\usepackage{axodraw} 
\usepackage{xspace} 
\usepackage[toc,page]{appendix} 
\usepackage[all,cmtip]{xy} 
\usepackage{relsize} 
\usepackage{shuffle} 
\makeatletter
\newcommand{\globalcolor}[1]{%
	\color{#1}\global\let\default@color\current@color
}
\makeatother

\usetikzlibrary{calc}
\usetikzlibrary{decorations}
\usetikzlibrary{positioning}
\usetikzlibrary{shapes}
\usetikzlibrary{external}

\definecolor{blush}{rgb}{0.87, 0.36, 0.51}
\definecolor{brightcerulean}{rgb}{0.11, 0.67, 0.84}
\definecolor{greenryb}{rgb}{0.4, 0.69, 0.2}

\newif\ifdark
\darkfalse

\ifdark
\definecolor{darkred}{rgb}{0.9,0.2,0.2}
\definecolor{darkblue}{rgb}{0.7,0.3,1}
\definecolor{darkgreen}{rgb}{0.1,0.9,0.1}
\definecolor{franck}{rgb}{0,0.8,1}
\definecolor{pagebackground}{rgb}{.15,.21,.18}
\definecolor{pageforeground}{rgb}{.84,.84,.85}
\pagecolor{pagebackground}
\AtBeginDocument{\globalcolor{pageforeground}}
\definecolor{symbols}{rgb}{0,0.7,1}
\colorlet{connection}{red!80!black}
\colorlet{boxcolor}{blue!50}

\else

\definecolor{darkred}{rgb}{0.7,0.1,0.1}
\definecolor{darkblue}{rgb}{0.4,0.1,0.8}
\definecolor{darkgreen}{rgb}{0.1,0.7,0.1}
\definecolor{franck}{rgb}{0,0,1}
\definecolor{pagebackground}{rgb}{1,1,1}
\definecolor{pageforeground}{rgb}{0,0,0}
\colorlet{symbols}{blue!90!black}
\colorlet{connection}{red!30!black}
\colorlet{boxcolor}{blue!50!black}

\fi

\def\slash{\leavevmode\unskip\kern0.18em/\penalty\exhyphenpenalty\kern0.18em}
\def\dash{\leavevmode\unskip\kern0.18em--\penalty\exhyphenpenalty\kern0.18em}

\DeclareMathAlphabet{\mathbbm}{U}{bbm}{m}{n}

\DeclareFontFamily{U}{BOONDOX-calo}{\skewchar\font=45 }
\DeclareFontShape{U}{BOONDOX-calo}{m}{n}{
	<-> s*[1.05] BOONDOX-r-calo}{}
\DeclareFontShape{U}{BOONDOX-calo}{b}{n}{
	<-> s*[1.05] BOONDOX-b-calo}{}
\DeclareMathAlphabet{\mcb}{U}{BOONDOX-calo}{m}{n}
\SetMathAlphabet{\mcb}{bold}{U}{BOONDOX-calo}{b}{n}

\setlist{noitemsep,topsep=4pt,leftmargin=1.5em}

\DeclareMathAlphabet{\mathbbm}{U}{bbm}{m}{n}

\DeclareMathAlphabet{\mcb}{U}{BOONDOX-calo}{m}{n}
\SetMathAlphabet{\mcb}{bold}{U}{BOONDOX-calo}{b}{n}
\DeclareFontFamily{U}{mathx}{\hyphenchar\font45}
\DeclareFontShape{U}{mathx}{m}{n}{
	<5> <6> <7> <8> <9> <10>
	<10.95> <12> <14.4> <17.28> <20.74> <24.88>
	mathx10
}{}
\DeclareSymbolFont{mathx}{U}{mathx}{m}{n}
\DeclareMathSymbol{\bigtimes}{1}{mathx}{"91}

\providecommand{\figures}{false}
{ \ifthenelse{\equal{\figures}{false}} {#1}{\[ {\rm Figure \ missing !} \]} }{}

\def\msH{\mathscr{H}}

\def\CT{\mathcal{T}}

\tikzstyle{tinydots}=[dash pattern=on \pgflinewidth off \pgflinewidth]
\tikzstyle{superdense}=[dash pattern=on 4pt off 1pt]

\newcommand{\mcO}{\mathcal{O}}

\newcommand{\mcR}{\mathcal{R}}
\newcommand{\mcC}{\mathcal{C}}

\newcommand{\mcF}{\mathcal{F}}

\newcommand{\mcG}{\mathcal{G}}

\newcommand{\beq}{\begin{equation}}
	\newcommand{\eeq}{\end{equation}}

\usepackage{empheq}

\newcommand{\mbbZ}{\mathbb{Z}}

\newcommand{\mbbR}{\mathbb{R}}

\def\${|\!|\!|}

\def\grad{\nabla}
\def\lap{\grad}

\def\pa{\partial}

\def\rbr#1{{\left(#1\right)}}

\def\set#1{{\{#1\}}}
\def\abs#1{{\left|#1\right|}}
\def\norm#1{{\left|\left|#1\right|\right|}}

\newcounter{theorem}

\newtheorem{ex}[theorem]{Example}

\newenvironment{DIFnomarkup}{}{} 

\newtheorem{assumption}{Assumption}

\theorembodyfont{\rmfamily}

\newcommand{\rrightarrow}{{\to\hskip -4.9mm\raise 1pt\hbox{$\to$}}}

\newfont{\indic}{bbmss12}

\def\bk{\boldsymbol{k}}

\def\Nabla_#1{\nabla_{\!#1}}

\definecolor{oxfordblue}{rgb}{0.0, 0.13, 0.28}
\tikzstyle{leaf}=[circle, draw=black, fill=gray!20, inner sep = .1pt, font=\tiny, minimum size=3mm]
\tikzstyle{inner}=[circle, minimum size=1.5mm, fill=oxfordblue, inner sep =0]
\tikzstyle{F} = [thick, color=maroon, dotted,dash pattern=on 1pt off 1pt]
\tikzstyle{H0} = [color=cyan, dotted]
\tikzstyle{Poi} = [color=cyan]
\tikzstyle{H} = [thick, color=maroon]
\tikzstyle{Phi} = [thick, color=maroon, snake=snake,
segment length=5pt,
segment amplitude=1pt]
\tikzstyle{Phi-conj} = [thick, color=maroon, snake=snake,
segment length=5pt,
segment amplitude=1pt,dashed]
\tikzstyle{stoch-int} = [color=cyan, snake=zigzag,
segment amplitude=.2mm,
segment length=.5mm]
\tikzstyle{stoch-conj-int} = [color=cyan, snake=triangles,
segment amplitude=.25mm,
segment length=.8mm]
\definecolor{modebeige}{rgb}{0.59, 0.44, 0.09}
\definecolor{maroon}{rgb}{0.5, 0.0, 0.0}
\definecolor{mayablue}{rgb}{0.45, 0.76, 0.98}

\let\d\partial

\def\eqref#1{(\ref{#1})}

\makeatletter 
\newcommand*{\bigcdot}{}
\DeclareRobustCommand*{\bigcdot}{%
	\mathbin{\mathpalette\bigcdot@{}}%
}
\newcommand*{\bigcdot@scalefactor}{.5}
\newcommand*{\bigcdot@widthfactor}{1.15}
\newcommand*{\bigcdot@}[2]{%
	\sbox0{$#1\vcenter{}$}
	\sbox2{$#1\cdot\m@th$}%
	\hbox to \bigcdot@widthfactor\wd2{%
		\hfil
		\raise\ht0\hbox{%
			\scalebox{\bigcdot@scalefactor}{%
				\lower\ht0\hbox{$#1\bullet\m@th$}%
			}%
		}%
		\hfil
	}%
}
\makeatother

\tcbset
{colframe=boxcolor,colback=symbols!7!pagebackground,coltext=pageforeground,
	fonttitle=\bfseries,nobeforeafter,center title,size=fbox,boxsep=1.5pt,
	top=0mm,bottom=0mm,boxsep=0mm,tcbox raise base}

\def\two{{\<generic>\kern0.05em\<genericb>}}
\def\twoI{{\<Ito>\kern0.05em\<Itob>}}

\def\mail#1{\burlalt{#1}{mailto:#1}}

\usepackage{thmtools} 

\declaretheorem[style=definition]{example}

\begin{document}
	
	\title{Birkhoff normal form via decorated trees}
	
	\author{Jacob Armstrong-Goodall$^1$, Yvain Bruned$^2$}
	\institute {Maxwell Institute for Mathematical Sciences, University of Edinburgh \and  Université de Lorraine, CNRS, IECL, F-54000 Nancy, France
		\\
		Email:\ \begin{minipage}[t]{\linewidth}
			\mail{j.a.armstrong-goodall@sms.ed.ac.uk} \\ \mail{yvain.bruned@univ-lorraine.fr}.
	\end{minipage}}
	
	\maketitle

	\begin{abstract}
		We derive an explicit tree based ansatz for the Birkhoff normal form up to any order in the context of Hamiltonian PDEs. To do so we make use of a tree based representation of iterated Poisson brackets to encode the nested Taylor expansions along flows of a sequence of symplectic transformations. As an example we consider the cubic Schrödinger equation.
	\end{abstract}

	\setcounter{tocdepth}{1}
	\tableofcontents

\section{Introduction}
The Birkhoff normal form is a classical technique in the study of Hamiltonian systems. Its purpose is to reduce the Hamiltonian to a simpler form via a series of symplectic transformations, leaving only resonant terms and a non-resonant remainder. Perhaps the most prevalent application of the Birkhoff normal form is in the development of KAM theory, which seeks to establish the long-term existence of quasi-periodic solutions under small perturbations in the energy. This problem was posed by Kolmogorov in 1954 and proved independently by Moser and Arnol'd in \cite{Moser62} and \cite{Arnol63}, respectively. The extension to infinite dimensions took place throughout the '80s and '90s in various works by Kuksin, Pöschel, Craig, Wayne, Bourgain, and many others \cite{Bou98,CraWa93,Kuksi87,KP96,Wayne90}. A generalization of the Birkhoff normal form theorem to Hamiltonian PDEs was performed by Bambusi in \cite{Bam03}, where it is proved that there exists a canonical transformation which reduces the Hamiltonian to its normal form up to order (N). This was followed up in work with Grébert \cite{BG06}, which demonstrates the applicability of this technique to a wide range of PDEs with a tame modulus condition. In \cite{BerFG20}, the stability result of \cite{BG06} was extended to a general class of fully resonant Schrödinger equations by introducing the rational normal form reduction. For a survey on the topic of normal forms in the context of PDE, see \cite{Grebe06}. Since then, applications in the study of PDEs include demonstrations of global well-posedness \cite{BeFaG20,Bou04,CoSYX25}, bounds on the growth of higher Sobolev norms \cite{CKO12}, long-term behavior of small solutions \cite{BG22}, and numerical analysis \cite{BeBCE23}. Another valuable contribution is found in \cite{Pere03}, where the convergence of the normal form itself is studied and shown to be either convergent or generically divergent. This paper also provides a good overview of the history of the Birkhoff normal form in perturbative Hamiltonian mechanics. In the related paper \cite{Kriko22}, Krikorian proves that the Birkhoff normal form is generally divergent, determining which of the two alternatives of Pérez-Marco's theorem in \cite{Pere03} is true.

Another well-known technique in the study of Hamiltonian systems is the Poincaré-Dulac normal form. This technique is more flexible than the Birkhoff variety studied here because it can be applied to Hamiltonian systems where the energy is not conserved. This flexibility comes at the expense of tracking information about the geometry of the space, meaning Poincaré-Dulac offers only a local picture of the dynamics. The Poincaré-Dulac approach is widely used to establish local well-posedness, for example in \cite{GKO13}, and to study the quasi-invariance of Gaussian measures via the modified energy in \cite{OST18,ST23}. The former work motivated the decorated tree approach by one of the authors in \cite{B24}, where the process of eliminating non-resonant interactions is connected to the formalism first introduced in the context of numerical schemes \cite{BS}, providing an algebraic perspective. Iterated Lie brackets appear in the Magnus expansion \cite{Magnu54}, for which a practical recursive algorithm similar in spirit to the present work can be found in \cite{Iserl11,IseNo99}. Algebraic aspects of the tree formulations for Magnus expansions and Lie integrators can be found in \cite{CuEbO20} and \cite{MunWr08}, respectively. The paper \cite{LauCo2025} uses a similar tree-based methodology to \cite{B24}, along with the Butcher-Connes-Kreimer Hopf algebra to study the group of Poisson diffeomorphisms, which leads to a new class of Poisson integrators. The tree-based approaches in \cite{B24,GKO13} were a catalyst for the approach we take here.

In the present work, we introduce an explicit formula for the Birkhoff normal form using decorated trees. The formulation of these trees is distinct from those in \cite{B24,BS} and is uniquely suited to describe Taylor expansions along Hamiltonian flows. The approach involves defining the symplectic transformations recursively and encoding the structure of the Poisson brackets into trees, resulting in a relatively intuitive formulation of the Birkhoff normal form to any order. One of the benefits of this approach is its generality; although the Schrödinger equation is used as an example, the rules for generating the correct trees are independent of the equation. This work may also pave the way for a deeper understanding of the combinatorial structure of resonances in Hamiltonian PDEs. It is somewhat surprising that no explicit form for such a widely used technique has appeared prior, but to our knowledge, this is the first presentation of such in the literature.

Let us briefly introduce this formalism.
One starts with an Hamiltonian of the form:
\begin{equation*}
	H = H_0 + H_1
	\end{equation*}
	where in the case of NLS, $H_0$ is a monomial of degree $2$ and $H_1$ is a monomial of degree $4$. The first step is to decompose $H_1$ into
	\begin{equation*}
		H_1 = H_1^{\text{\tiny{res}}} + H_1^{\text{\tiny{non-res}}}
		\end{equation*}
where  $H_1^{\text{\tiny{res}}}$ is the resonant part and  $H_1^{\text{\tiny{non-res}}}$ is the non-resonant part. One gets
\begin{equation*}
	H = H_0 + H_1^{\text{\tiny{res}}} + H_1^{\text{\tiny{non-res}}}.
	\end{equation*}
	Then, one wants to compose $H$ with a symplectic transform $F$ for eliminating $H_1^{\text{\tiny{non-res}}}$ and get higher resonant terms. It means that
	\begin{equation*}
		H \circ F = H+\sum_{n=1}^N\frac{\set{H,F}^{n}}{n!} + R_{N+1}
	\end{equation*}
	with some remainder $R_{N+1}$ and we ask $F$ to satisfy
	\begin{equation*}
		\set{H_0,F} =  - H_1^{\text{\tiny{non-res}}}.
		\end{equation*}
		This allows us to make disappear the term $H_1^{\text{\tiny{non-res}}}$. This procedure can be iterated with successive $F_i$ and it produces  the Birkhoff normal form. The main result of this paper, Theorem \ref{main_theorem}, provides an explicit formula of this decomposition with decorated trees:
		 One has for $m, \ell \in \mathbb{N}^{*}$ with $ m < \ell$.
			\begin{equation*}
				\begin{aligned}
					\msH_m^{\ell}=(H\circ F_1 \circ \cdots \circ F_m)^{\ell} & =  H_0 + \sum_{T \in \, \mathcal{T}^{\! < m + 2}_r}  \frac{(\Pi T)}{S(T)} +  \sum_{T \in \, \mathcal{T}^{\! m + 2}_{\circ}}  \frac{(\Pi T)}{S(T)} \\ & + \sum_{T \in \, \mathcal{T}_{\circ}^{m + 2,\ell}}   \frac{(\Pi T)}{S(T)}
				\end{aligned}
			\end{equation*}
			Moreover, one has
			\begin{equation*}
				F_i =  \sum_{T \in \, \mathcal{T}^{\! i+1}_{n}} \frac{(\Pi T)}{S(T)}, \quad \set{H_0, F_i} =  - \sum_{T \in \, \mathcal{T}^{\! i+1}_{\circ}} \frac{(\Pi T)^{\text{\tiny{non-res}}}}{S(T)},
			\end{equation*}
		where
		\begin{itemize}
			\item $ \msH^\ell_m $ is the $ m$-th order normal form truncated at order $ 2\ell $.
			\item  $\mathcal{T}^{\! m }_{\circ}$ (resp. $\mathcal{T}^{\! m,\ell }_{\circ}$) are decorated trees encoding iterated brackets where the outer bracket is $\set{\cdot, \cdot}$ and with the size $p$ of the monomials to be  equal to $ 2m $ (resp. $ 2m < p \leq 2 \ell $).
			\item $\mathcal{T}^{\! < m}_{r} $ encodes monomials of size strictly smaller than $2 m$ with outer bracket $\set{\cdot, \cdot}^{\text{\tiny{res}}}$. This means that one keeps only the resonant part of the iterated bracket.
			\item $\mathcal{T}^{\! m}_{n}$
			encodes monomials of size  $2 m$ with outer bracket $\set{\cdot, \cdot}^{\text{\tiny{non-res}}}_{\Phi}$. This means that one keeps only the non-resonant part of the iterated bracket divided by a well-chosen phase.
			\item $ F_i $ is an order $ 2i+2 $ symplectic Hamiltonians, \item The operator $ \Pi $ maps the tree back to a term consisting of iterated Poisson brackets while $ S(T) $ corresponds to the coefficient of the Taylor expansion.
			\end{itemize}
			 This theorem presents an explicit expression for the normal form up to any order. It is proved via induction. Below, we provide some examples of decorated trees:
			 \begin{equation*}
			 	\begin{aligned}
			 	 ( \Pi \, \begin{tikzpicture}[node distance=14pt, baseline=5]
			 		\node[leaf] (leaf) at (0,0.3) {$r$};
			 	\end{tikzpicture} ) &= H_1^{\text{\tiny{res}}} , \quad  (\Pi \, \begin{tikzpicture}[node distance = 14pt, baseline=10]
			 		\node[leaf] (inner) at (0,0.3) {$\circ $};
			 		\node[leaf] (leaf1) [above left of = inner] {$ \circ $};
			 		\node[leaf] (leaf2) [above right of = inner] {$ n $};
			 		\draw[H] (inner.north west) -- (leaf1.south east);
			 		\draw[H] (inner.north east) -- (leaf2.south west);
			 	\end{tikzpicture} ) =  	\set{H_1, H_{1,\Phi}^{\text{\tiny{non-res}}}}, \\ (\Pi \, \begin{tikzpicture}[node distance = 14pt, baseline=10]
			 		\node[leaf] (inner) at (0,0.3) {$\circ $};
			 		\node[leaf] (leaf1) [above left of = inner] {$ \circ $};
			 		\node[leaf] (leaf3) [above left of = leaf1] {$ k $};
			 		\node[leaf] (leaf4) [above right of = leaf1] {$ n $};
			 		\node[leaf] (leaf2) [above right of = inner] {$ n $};
			 		\draw[H] (inner.north west) -- (leaf1.south east);
			 		\draw[H] (inner.north east) -- (leaf2.south west);
			 		\draw[H] (leaf1.north east) -- (leaf4.south west);
			 		\draw[H] (leaf1.north west) -- (leaf3.south east);
			 	\end{tikzpicture}) & = \set{\set{H_0,H_{1,\Phi}^{\text{\tiny{non-res}}}}, H_{1,\Phi}^{\text{\tiny{non-res}}}}.
			 	\end{aligned}
			 \end{equation*}
			 The decorated trees are planar binary trees with node decorations $k, \circ, n, r$ that encode the various monomials and resonant/non-resonant parts.
	
In Section \ref{sec::Hamiltonian-Form}, we briefly introduce Hamiltonian systems, the symplectic form and the Poisson bracket in canonical coordinates. We demonstrate how this leads to a concept of Taylor expansion of a function along the flow of the Hamiltonian, which is an essential tool in constructing the normal form. Next, we detail the Hamiltonian of the cubic Schr\"odinger equation, explaining its structure in Fourier space, including the resonance conditions and phase functions. In Section \ref{sec::Birkhoff-Normal-Form}, we describe the algorithm for normal form reduction. This iterative procedure involves splitting each term in the Hamiltonian into resonant and non-resonant components before introducing a symplectic transform to eliminate the lowest remaining non-resonant terms, producing a new Hamiltonian. Each iteration introduces an infinite number of higher-order terms to be addressed in later iterations. The goal of this section is to elucidate the structure of these terms using the language of Poisson brackets and introduce definitions relevant to the tree based proofs in \ref{sec::tree-formulation} such as \ref{def::truncation-term}. We conclude Section \ref{sec::Birkhoff-Normal-Form} by introducing the recursive formulation of the symplectic transformations used in the normal form reduction. This formulation allows us to encode the entire normal form as trees generated according to certain rules. The construction and explanation of the decorated tree formulation is the focus of Section \ref{sec::tree-formulation}. We begin by introducing the notation and carefully describing the rules governing the structure of trees that may occur in the normal form. This is performed via Assumption \ref{assumption_1}. Proposition \ref{prop_iterated_bracket} establishes the mapping between the trees and the Taylor coefficients occurring in the iterated Poisson brackets, which is non-trivial because higher-order terms may contain lower-order Taylor expansions along multiple Hamiltonian flows. In Definition \ref{spaces_trees}, we define the spaces where all possible subtrees must reside according to the previously described rules. We finish by stating the main result of this paper, Theorem \ref{main_theorem} and we conduct an inductive proof.

\subsection*{Acknowledgements}

{\small
	Y. B. gratefully acknowledges funding
	support from the European Research Council (ERC) through the ERC Starting Grant Low
	Regularity Dynamics via Decorated Trees (LoRDeT), grant agreement No. 101075208.
	Views and opinions expressed are however those of the author(s) only and do not necessarily
	reflect those of the European Union or the European Research Council. Neither the European
	Union nor the granting authority can be held responsible for them. J. A. G. is supported
	by the EPSRC Centre for Doctoral Training in Mathematical Modelling, Analysis and
	Computation (MAC-MIGS) funded by the UK Engineering and Physical Sciences Research
	Council (grant EP/S023291/1), Heriot-Watt University and the University of Edinburgh.
} 
	
	\section{Hamiltonian Equations}\label{sec::Hamiltonian-Form}
	Consider the infinite-dimensional phase space $ M $ with coordinates $ \mu_k $ and momenta $ \nu_k $, for $ k\in\mbbZ^d $. Then the Hamiltonian is a function $ H:M\rightarrow \mbbR $ and the solution to the Hamiltonian system is a curve $ \rbr{\mu_k(t),\nu_k(t)} $ in $ M $ which obeys the Hamiltonain equations of motion
	\begin{equation}\label{eq:geo-struc-det}
		\dot{\mu}_k = \frac{\pa H}{\pa \nu_k},\quad \dot{\nu}_k = \frac{\pa H}{\pa \mu_k}.
	\end{equation}  
	When the Hamiltonian is time independent, $ \pa H/\pa t = 0  $, the value of the function $ H(\mu,\nu) $ is conserved over the evolution of the solutions of the equation. This conservation law can also be expressed via the symplectic form,
	\begin{equation}\label{eq:symplectic-form}
		\omega(\mu,\nu)=\int d\mu(t)\wedge d\nu(t)dx = \int d\mu_0\wedge d\nu_0 dx,\quad\forall t\geq0.
	\end{equation}
	The pair $ (M,\omega) $ is a symplectic manifold. This identity corresponds to the conservation of energy in physical systems.
	\begin{remark}
		The evenness of the dimensionality of the phase space is a corrolary of the well know theorem of Darboux which states that any point of $ M $ has local coordinates $ (\mu,\nu) = \rbr{\set{\mu_k}_{k\in\mbbZ^d},\set{\nu_k}_{k\in\mbbZ^d}} $, in which the symplectic structure has the form $ \omega=\sum_{k}d\mu_k\wedge\d\nu_k $. These are called the symplectic or canonical coordinates.
	\end{remark} 
	For a smooth function $ f\in\mcC^{\infty}(M) $ the symplectic form associates a smooth vector field $ X_f $ on $ M $ implicitly defined by $ \omega(X_f,\cdot) =df $ where $ df $ is the exterior derivative of $ f $.
	This definition allows us to introduce the Poisson bracket as an operation on functions, that is, as the symplectic form $ \omega(\mu,\nu) = \sum_{k}\mu_k\wedge \nu_k $ between two Hamiltonian vector fields $ X_F $ and $ X_G $
	\begin{equation*}
		\set{F,G} = \omega(X_F,X_G).
	\end{equation*}
	\begin{definition}
	The Poisson bracket, defined as a differential operator on functions $ f,g:M\rightarrow \mbbR $, is given in canonical coordinates by: 
	\begin{equation*}\label{def:poisson-brackets} \set{f,g} = i \sum_{k\in\mbbZ^d} \left( \partial_{\mu_k} f \partial_{\nu_k} g - \partial_{\mu_k} g \partial_{\nu_k} f  \right) \end{equation*}
	\end{definition}
	Furthermore, a Hamiltonian vector field $ X_F $ can be applied as a directional derivative to a sufficiently smooth function $ g $
	\begin{equation*}
		X_F(g) = \set{g,F}.
	\end{equation*}
	Specifically the vector field $ X_H $ associated to a Hamiltonian $ F $ is defined as 
	\begin{equation*}
		X_F = \rbr{\frac{\pa H}{\pa \mu_k},-\frac{\pa F}{\pa \nu_k}} = JdF
	\end{equation*}
	where  \[J = \begin{pmatrix} 0 & I_k \\ -I_k & 0 \end{pmatrix},\quad dF = \begin{pmatrix}
		\frac{\pa F}{\pa \mu_k} \\ \frac{\pa F}{\pa \nu_k}.
	\end{pmatrix} \]
	In this formulation for example if $ F=\mu_k $ then $ X_F = \frac{\pa }{\pa\nu_k} $
	This allows us to write the following definition:
	\begin{definition}[Hamiltonain Taylor Expansion]
	We define the Taylor expansion of $ g $ along the flow generated by the Hamiltonian vector field by
	\begin{equation*}
		g\circ F = g+\sum_{n=1}^N\frac{\set{g,F}^{n}}{n!} + R_{N+1}
	\end{equation*}
	where the remainder term is,
	\begin{equation}
		R_n(t) = \int_0^1\frac{\set{f,H}^{n+1}}{(n+1)!}(1 - s)^{n+1}ds.
	\end{equation}
	\end{definition}
	\begin{notation}
	Here we use the notation
	\begin{equation*}
		\set{X,Y}^n = \set{\set{\dots\set{X,Y},\dots, Y},Y}
	\end{equation*}
	where the right hand side has $ n $ nested brackets.
	\end{notation}
	Next we will introduce the cubic Schr\"odinger equation as an example of a Hamiltonain PDE exhibiting a conserved symplectic structure
	\subsection{The Cubic Schr\"odinger Equation}
	As a running example, we consider the defocusing cubic Schrödinger equation on the $d-$dimensional torus $\mathbb{T}^d$ given by:
	\begin{align} 
		\begin{cases}
			i \partial_t u + \Delta  u  = | u |^{2} u \\
			u |_{t = 0} = u_0,
		\end{cases}
		\quad (x, t) \in \mathbb{T}^d \times \mathbb{R}_+.
		\label{NLS1}
	\end{align}
	This equation has two conserved quantities. The first is the mass described by the $ L^2 $ norm
	\begin{equation*}\label{eq::mass}
		\norm{u}_{L^2}^2 = \sum_{k\in\mbbZ^d}|u_k|^2.
	\end{equation*}
	The second conserved quantity is the Hamiltonian which in physical space is: 
	\begin{equation*}
		\mathscr{H}_0 = \frac12 \norm{ \lap u}^2_{L^2} + \frac14||u||^4_{L^2},
	\end{equation*}
	which can likewise be written under the Fourier transform as 
	\begin{equation}\label{eq::Hamiltonian}
		\msH_0(u,\bar{u})=\frac{i}{2}\sum_{k\in\mbbZ^d}\abs{k}^2\abs{u_k}^2+\frac{i}{4}\sum_{\substack{\bk \in (\mbbZ^{d})^4 }}u_{k_1}\bar{u}_{k_{2}}u_{k_3}\bar{u}_{k_4}.
	\end{equation}	 
	For the Schr\"odinger equation with integer nonlinearity the Fourier coefficients introduced by the nonlinear term in the Hamiltonian are zero unless  $ k_1-k_2+k_3-k_4=0 $. This is the result of the Fourier transform 
	\begin{equation*}
		\mcF u = \sum_{k\in\mbbZ^d}u_ke^{ikx},\quad u_k = \frac{1}{2 \pi} \int_0^{2 \pi} u e^{-ikx}dx.
	\end{equation*}
	From this the nonlinear or interaction part of the Hamiltonian becomes
	\begin{equation*}
		H_1 = \frac{i}{4}\sum_{\substack{\bk \in(\mbbZ^{d})^4 }}\mcG_{1234}u_{k_1}\bar{u}_{k_{2}}u_{k_3}\bar{u}_{k_4}
	\end{equation*}
	where 
	\begin{equation*}
		\mcG_{1234}=\int_0^{2\pi}{e^{i(k_1-k_2+k_3-k_4)x}}dx=\begin{cases}2\pi &k_1-k_2+k_3-k_4=0 \\ 0 &\text{otherwise.}\end{cases}
	\end{equation*}
	From now on for clarity of presentation, we will assume this condition to hold in the case of our Hamiltonain, and later the symplectic transforms introduced to eliminate non-resonant terms.  The Hamiltonian \ref{eq::Hamiltonian} allows us to write the nonlinear Sch\"odinger equation in its Hamiltonain form as
	\begin{equation}\label{eq:HamiltonianFormulation}
		\partial_t u_k = \pa_{\bar{u}_k} \msH_0(u, \bar{u}), \quad 	\partial_t \bar{u}_k = - \pa_{u_k}\msH_0(u,\bar{u}).
	\end{equation}
	We denote the phase function on $ \bk \in (\mbbZ^{d})^{2n} $ for a given term in the Hamiltonian by
	\begin{equation*}
		\Phi_n(\bk) = \sum_{i=1}^{2n}(-1)^{i-1}k_i^2.
	\end{equation*}
	For instance the phase function corresponding to the quartic term in the Hamiltonian \ref{eq::Hamiltonian} is 
	\begin{equation*}
		\Phi_2(\bk) = k_1^2 -k_2^2 + k_3^2 -k_4^2.
	\end{equation*}
	We define the set of frequencies for which this is `nearly resonant' by
	\begin{equation*}
		\mcR_{n;N} = \set{\boldsymbol{k}\in (\mbbZ^d)^{2n}: \abs{\Phi_n(\bk)} \leq N}.
	\end{equation*}
	A special case of this is when the phase is identically zero. In that case we have 
	\begin{equation*}
		\Phi_2(\bk)=0,\quad k_1-k_2+k_3-k_4=0,
	\end{equation*}	
	meaning that the terms must be paired together. We do not restrict ourselves to this setting because almost-resonance as defined above is very useful, for instance when bounding Sobolev norms in \cite{CKO12}. By convention we will just refer to resonant terms when we mean almost resonant, with `purely resonant' as a special case. The indicator function over the resonant frequencies $ \mcR_n $ is written
	\begin{equation*}
		\one_{\mcR_{n}}(\bk;N) = \begin{cases}1,\quad \abs{\Phi_n(\bk)} \leq N \\ 0,\quad \abs{\Phi_n(\bk)}> N.\end{cases}
	\end{equation*}
	\begin{definition}\label{def::filter-function}
		With the phase and indicator function defined as above we will introduce functions of the following form to filter out resonant terms in the Birkhoff normal form,
		\begin{equation}
			\mcF_{n,N}(\bk) = \frac{\one_{\mcR_n^c}(\bk)}{2\Phi_n(\bk)} = \begin{cases}\frac{1}{2\Phi_n(\bk)},\quad &\abs{\Phi_n(\bk)} > N \\ 0,\quad &\text{otherwise.}\end{cases}
		\end{equation}
	\end{definition}
	These definitions provide us with the ingredients required to explain the Birkhoff normal form in the case of \eqref{NLS1}. This process easily generalises to other Hamiltonian PDEs. 
	\section{The Birkhoff Normal Form Reduction}\label{sec::Birkhoff-Normal-Form}
	We have the Hamiltonian 
	\begin{equation*}
		\msH_0=H_0 + H_1 = \frac{i}{2}\sum_{k\in\mbbZ^d}\abs{k}^2\abs{u_k}^2+\frac{i}{4}\sum_{\bk\in(\mbbZ^{d})^4}u_{k_1}\bar{u}_{k_{2}}u_{k_3}\bar{u}_{k_4}.
	\end{equation*}
	The algorithm works by splitting $ H_1 $ into resonant part $	H_1^{\text{\tiny{res}}}$  and non-resonant part  $	H_1^{\text{\tiny{non-res}}}$ before composing with a symplectic transform that removes the non-resonant terms. For the first iteration we write
	\begin{equation*}
		H_1 = H_1^{\text{\tiny{res}}} + H_1^{\text{\tiny{non-res}}},
	\end{equation*}
	where
		\begin{equation*}
			\begin{aligned}
		H_1^{\text{\tiny{res}}} &= \frac{i}{4} \sum_{\substack{\bk\in (\mbbZ^{d})^4\\ \Phi_2(\bk)\leq N}} u_{k_1}\bar{u}_{k_{2}}u_{k_3}\bar{u}_{k_4}, \\	H_1^{\text{\tiny{non-res}}} &= \frac{i}{4} \sum_{\substack{\bk\in (\mbbZ^{d})^4\\  \Phi_2(\bk) > N}} u_{k_1}\bar{u}_{k_{2}}u_{k_3}\bar{u}_{k_4},
		\end{aligned}
	\end{equation*} 
	and
	\begin{equation*}
	\Phi_2(\bk) = k_1^2 - k_2^2 + k_3^2 - k_4^2. 
	\end{equation*}
	The complete Hamiltonian can now be written as
	\begin{equation}\label{eq::res-decomp-Hamiltonian}
		\msH_0 =\frac{i}{2}\sum_{k\in\mbbZ^{d}}\abs{k}^2\abs{u_k}^2+ \frac{i}{4} \sum_{\substack{\bk\in(\mbbZ^{d})^4\\ \Phi_2(\bk)>N}} u_{k_1}\bar{u}_{k_{2}}u_{k_3}\bar{u}_{k_4} + \frac{i}{4} \sum_{\substack{\bk\in(\mbbZ^{d})^4\\ \Phi_2(\bk)\leq N}} u_{k_1}\bar{u}_{k_{2}}u_{k_3}\bar{u}_{k_4}.
	\end{equation}
	We then introduce a symplectic transform $ \Gamma_{F_1} $ such that
	\begin{equation*}  \set{H_0,F_1} = -H_1^{\text{\tiny{non-res}}}.
		\end{equation*}
		 Such a transform can be written as
	\begin{align}\label{eq::F1}
		F_1 &=-\sum_{\bk\in(\mbbZ^{d})^4}\mcF_{2,N}(\bk)u_{k_1}\bar{u}_{k_2}u_{k_3}\bar{u}_{k_4}\notag\\ &:= -H_{1,\Phi}^{\text{\tiny{non-res}}}.
	\end{align} 
	The second line is a convenient notation for the the first which follows from definition \ref{def::filter-function} with,
	\begin{equation*}
		\mcF_{2,N}(\bk) =  \frac{\one_{\mcR_{2,N}^c}(\bk)}{2 \Phi_2(\bk)}.
	\end{equation*}
	\
		It's easy to check that with $ F_1 $ defined above the identity $ \set{H_0,F_1} = -H_1^{\text{\tiny{non-res}}} $ holds. One has
		\begin{align}\label{eq::F1}
			\set{H_0, F_1} &=  -\frac{i}{2}\sum_{k \in \mbbZ^{d}} \sum_{\bk\in(\mbbZ^{d})^4}\mcF_2(\bk)\left( \pa_{u_{k}} (|k|^2 |u_{k}|^2) \pa_{\bar u_{k}}(u_{k_1}\bar{u}_{k_2}u_{k_3}\bar{u}_{k_4}) \right. \notag\\ & \left. - \pa_{\bar{u}_{k}} (|k|^2 |u_{k}|^2) \pa_{ u_{k}}(u_{k_1}\bar{u}_{k_2}u_{k_3}\bar{u}_{k_4}) \right) \notag\\
			&= - \frac{i}{2} \sum_{\bk\in (\mbbZ^{d})^4}(k_1^2 - k_2^2 + k_3^2 - k_4 ^2)\mcF_2(\bk)u_{k_1}\bar{u}_{k_2}u_{k_3}\bar{u}_{k_4} \notag\\
			&= -  \frac{i}{4}\sum_{\bk\in (\mbbZ^{d})^4}\one_{\mcR_{2,N}^c}(\bk)u_{k_1}\bar{u}_{k_2}u_{k_3}\bar{u}_{k_4} \notag \\ 
			&= -  \frac{i}{4} \sum_{\substack{\bk\in (\mbbZ^{d})^4\\ \Phi_2(\bk)>N}}u_{k_1}\bar{u}_{k_2}u_{k_3}\bar{u}_{k_4} = -H_1^{\text{\tiny{non-res}}}.
		\end{align}
		When we write $ F_1 = H_{1,\Phi}^{\text{\tiny{non-res}}} $ the subscript of $ \Phi $ refers to the phase function introduced by the laplacians in the kinetic part of the Hamiltonain in the Poisson bracket $ \set{H_0,F_1} $, which we will see used frequently when defining $ F_i $. 

	The Birkhoff normal form is constructed as a sequence of Taylor expansions along the flows of each $ F_i $, which by construction leaves only resonant terms and a non-resonant remainder. We will now introduce a norm for quantifying the order of terms in the Hamiltonain Taylor expansion which is introduced in \cite{KP96}. 
	\begin{definition}[Sobolev Norm on Sequences]\label{def::norm}
		Define the space $ \ell^{2,p}_b $ to be the set of bi-infinite sequences such that
		\begin{equation*}
			\norm{u}_{a,p}^2 = \abs{q_0}^2+\sum_{k\in \mbbZ^{d} }\abs{u_k}^2\abs{k}^{2p}e^{2\abs{k}a}<\infty.
		\end{equation*} 
	\end{definition}
	We also make use of \cite[Lemma 2]{KP96} which is the following bilinear estimate.
	\begin{lemma}
		For $ a\geq 0 $ and $ p\geq\frac12 $, the space $ \ell^{a,p}_b $ is a Hilbert algebra with respect to the convolution of sequences, and 
		\begin{equation*}
			\norm{u\star u'}_{a,p}\leq c\norm{u}_{a,p}\norm{u'}_{a,p},
		\end{equation*}
		where $ c $ depends on $ p $ only.
	\end{lemma}	
	In this norm the interaction Hamiltonian of the cubic Schr\"odinger equation is $ H_1=\mcO(\norm{u}_{a,p}^4) $. There are other norms on iterated Poisson brackets such as the one used by Bourgain in \cite[Section 3]{Bou04}, used also in \cite[Equation 2.12]{CKO12}. 
	By composing $ H $ with this transform we obtain the first iteration of the Birkhoff normal form,
	\begin{align*}
		\msH_1&=H\circ F_1 \\
		&= H_0 + H_1 + \set{H_0,F_1} + \set{H_1, F_1} + \frac12\set{\set{H_0, F_1},F_1} \\
		&\quad + \int_0^1(1-s)\set{\set{H_1,F_1},F_1}ds \\
		&= H_0 + H_1^{\text{\tiny{res}}} + \set{H_1, F_1} + \frac12\set{\set{H_0, F_1},F_1}  + \mcO(\norm{u}^8_{a,p})
	\end{align*}
	where we have constructed the desired cancellation 
	\begin{equation*}
		\set{H_0,F_1} = -H_1^{\text{\tiny{non-res}}}.
	\end{equation*}
	Now we look at the lowest order terms which still contains non-resonant frequencies which are
	\begin{equation*}
		\set{H_1, F_1},\quad \frac12\set{\set{H_0,F_1},F_1}.  
	\end{equation*}
	We introduce a second symplectic transform $ \Gamma_{F_2} $ characterised by the property that 
	\begin{align*}
		-\set{H_0,F_2} &=\set{H_1,F_1}^{\text{\tiny{non-res}}}+\frac12\set{\set{H_0,F_1},F_1}^{\text{\tiny{non-res}}} \\
		&= \set{H_1,H_{1,\Phi}^{\text{\tiny{non-res}}}}^{\text{\tiny{non-res}}}+\frac12\set{\set{H_0,H_{1,\Phi}^{\text{\tiny{non-res}}}},H_{1,\Phi}^{\text{\tiny{non-res}}}}^{\text{\tiny{non-res}}} 
	\end{align*}
	which gives
	\begin{equation}\label{eq::F2}
		F_2 = \set{H_1,H_{1,\Phi}^{\text{\tiny{non-res}}}}^{\text{\tiny{non-res}}}_{\Phi}+\frac12\set{\set{H_0,H_{1,\Phi}^{\text{\tiny{non-res}}}},H_{1,\Phi}^{\text{\tiny{non-res}}}}^{\text{\tiny{non-res}}}_{\Phi},
	\end{equation}
	where the subscript $ \Phi $ on the bracket is the $ \mcO(\norm{u}_{a,p}^6) $ phase introduced by the bracket of $ F_2 $ against $ H_0 $, while the one on the Hamiltonian in the bracket is phase corresponding to $ F_1 $ which is calculated in \eqref{eq::F1}. Composing with the new transform gives
	\begin{align*}
		\msH_2 &= \msH_1\circ F_2 \\
		&= H_0 + H_1  + \set{H_0, F_1} + \set{H_1, F_1} + \frac12\set{\set{H_0,F_1},F_1} + \set{H_0, F_2}\\
		&\quad + \frac12\set{\set{H_1,F_1},F_1} + \frac16\set{\set{\set{H_0,F_1},F_1},F_1} + \set{H_1,F_2} + \set{\set{H_0,F_1},F_2} \\
		&\quad + \int_0^1(1-s)\set{\set{H_1,F_1}, F_2}ds \\
		&= H_0 + H_1^{\text{\tiny{res}}} + \set{H_1,F_1}^\text{\tiny{res}} + \frac12\set{\set{H_1,F_1},F_1} + \set{H_1^\text{\tiny{res}},F_2}  \\
		&\quad + \mcO(\norm{u}_{a,p}^{10})
	\end{align*}
	Likewise, introducing a third symplectic transform $ \Gamma_{F_3} $ defined in order to remove the lowest order non-resonant frequencies in $ \msH_2 $ gives the Hamiltonian
	\begin{align*}
		\msH_3 &= \msH_2\circ F_3 \\
		&= H_0 + H_1^{\text{\tiny{res}}} + \set{H_1,F_1}^\text{\tiny{res}} + \set{H_1^\text{\tiny{res}},F_2}^\text{\tiny{res}}+\frac12\set{\set{H_1,F_1},F_1}^{\text{\tiny{res}}}  \\
		&\quad+\set{H_1^{\text{\tiny{res}}},F_3} + \set{\set{H_1,F_1}^{\text{\tiny{res}}},F_2} + \frac16\set{\set{\set{H_1,F_1}, F_1}, F_1} \\
		&\quad + \mcO(\norm{u}_{a,p}^{12}).
	\end{align*}
	Notice that the three leading terms are now resonant and each succeeding term has non-resonant components. At the $ nth $  iteration of the Birkhoff normal we remove order$-(2n+2) $ non-resonant terms implying that $ F_n \sim \mcO(\norm{u}_{a,p}^{2n+2})$. We wish to introduce notation for the remainder terms in the Birkhoff normal form.
	\begin{definition}[Truncation Term]\label{def::truncation-term}
		Let $ s_n^m $ be the set of non-decreasing sequences of natural numbers less than and including $ n $ whose terms sum to $ m $, and let $ c_{z} $ be a coefficient determined by the number of repetitions of a number in the sequence and $ q(z) $ the length of the sequence. Then
		\begin{equation*}
			R_n^m(\cdot) = \sum_{z\in s_n^m} \frac{\set{\cdot,F_{z}}^{q(z)}}{c_{z}}
		\end{equation*}
		represents the sum of all possible $ \mcO(\norm{u}_{a,p}^{r+2m}) $ Poisson brackets made up of symplectic transforms $ F_i $, where $ r $ is the order of the monomials in the Hamiltonian taking the place of $ \cdot $.
	\end{definition}
	\begin{remark}
		Definition \ref{def::truncation-term} formally describes the lowest order terms  still containing non-resonant frequencies but we don't know a priori the structure of $ F_i $ so while \ref{def::truncation-term} we haven't yet arrived at an explicit description of the terms in the normal form.
	\end{remark}
	\begin{example}
	 	If $ N=3 $ and $ n=3 $ then we have $ z_1 = (1,1,1)  $, $ z_2 = (1,2) $, and $ z_3 = (3) $ hence $ c_{z_1} = 3! $, $ c_{z_2} = c_{z_3} = 1 $ while $ q(z_1) = 3, q(z_2)=2,$ and $q(z_3)=1 $. This gives
		\begin{equation*}
			R_3^3(H) = \sum_{z\in s_3} \frac{\set{H,F_{z}}^{q(z)}}{c_z} = \frac16\set{\set{\set{H,F_1},F_1},F_1} + \set{\set{H,F_1},F_2}+\set{H,F_3}.
		\end{equation*}
		We can see that these are all of the $ \mcO(\norm{u}_{a,p}^8) $ and $ \mcO(\norm{u}_{a,p}^{10}) $ terms for $ H_0 $ and $ H_1 $ respectively.
	\end{example}
	\begin{definition}[Truncated Normal Form]
		We define $ \msH_n^m  $ to be the $ n$-th iteration of the Birkhoff normal form truncated at order $ 2m $. In general, if we denote by $ \msH_n^{\text{\tiny{res}}} $ the resonant terms in $ \msH_n $, which we know are at most $ \mcO(\norm{u}_{a,p}^{2+2n}) $, then we can write the following
		\begin{equation*}
			\msH_n = \msH_n^{\text{\tiny{res}}} + \sum_{m = n+1}^\infty R_n^m(H),
		\end{equation*}
		and further we can write a truncation independent of the iteration of the Birkhoff normal form
		\begin{equation*}
			\msH_n^\ell = \msH_n^{\text{\tiny{res}}} + \sum_{m = n+1}^{\ell} R_n^m(H).
		\end{equation*}
	\end{definition}
	With these definitions we can clarify the form of the symplectic transforms $ F_i $ used to eliminate the non-resonant terms. The particular structure of each term is recursive because it can  be defined as a combination of Hamiltonians and brackets occurring in previous terms. This recursive formula for $ F_i $ is motivated by the fact that 
	\begin{equation*}
		F_n = -\mcF_{n+1,N}(\bk) \rbr{R_{n-1}^{n}(H_0) + R_{n-1}^{n-1}(H_1)}
	\end{equation*}
	which is the sum of non-resonant terms of order $ 2+2n $. It's easy to verify that \eqref{eq::F1}, \eqref{eq::F2} and $ F_3 $ below satisfy this formula. 
	\begin{align*}
		F_3 &= \frac12\set{\set{H_1,F_1},F_1}_{\Phi}^{\text{\tiny{non-res}}} + \set{H_1^{\text{\tiny{res}}},F_2}_{\Phi}^{\text{\tiny{non-res}}} \\
		&\quad + \frac16\set{\set{\set{H_0,F_1}, F_1}, F_1}^{\text{\tiny{non-res}}}_{\Phi} +\set{\set{H_0,F_1}, F_2}^{\text{\tiny{non-res}}}_{\Phi},
	\end{align*}
	which by substitution becomes the following recursive expression:
	\begin{align*}
		F_3 &= \frac12\set{\set{H_1,H_{1,\Phi}^{\text{\tiny{non-res}}}},H_{1,\Phi}^{\text{\tiny{non-res}}}}^{\text{\tiny{non-res}}}_{\Phi} + \set{H_1^{\text{\tiny{res}}},\set{H_1,H_{1,\Phi}^{\text{\tiny{non-res}}}}_{\Phi}^{\text{\tiny{non-res}}}}_{\Phi}^{\text{\tiny{non-res}}} \\
		&\quad+\frac16\set{\set{\set{H_0,H_{1,\Phi}^{\text{\tiny{non-res}}}}, H_{1,\Phi}^{\text{\tiny{non-res}}}}, H_{1,\Phi}^{\text{\tiny{non-res}}}}^{\text{\tiny{non-res}}}_{\Phi}\\ &\quad +\set{\set{H_0,H_{1,\Phi}^{\text{\tiny{non-res}}}}, \set{H_1,H_{1,\Phi}^{\text{\tiny{non-res}}}}^{\text{\tiny{non-res}}}_{\Phi} }^{\text{\tiny{non-res}}}_{\Phi}.
	\end{align*}
	The implication is that we can write the entire reduced Hamiltonian to any order in terms of only $ H$ and $ H_{1,\Phi}^{\text{non-res}}$. This observation motivates us to encode the iterative structure of the above terms as trees. In the next section we will introduce a decorated tree formalism to encode these expansions to provide an explicit formula of the Birkhoff normal form to any order.
\section{Decorated trees formulation}\label{sec::tree-formulation}
In this section, we introduce a decorated tree formalism for encoding iterated Poisson Lie Brackets. The decorations will be only on the nodes and will allow us to describe resonant and non-resonant terms.

We suppose a finite set of decorations $ I = \set{\circ, k, n, r} $ and we consider decorated trees whose nodes are decorated by elements of $I$. Such a tree denoted by $T^{\mathfrak{n}}$ satisfies:
\begin{itemize}
	\item $T$ is a planar binary rooted tree  with root node $ \rho_T $, edge set $ E_T $ and node set $ N_T $. 
	\item The map $\mathfrak{n} : N_T \rightarrow I$ are node decorations.
	We assume several constraints on the node decorations:
	\begin{enumerate}[label=(\alph*)]
		\item A node decorated by $ n $ is connected to its parent in the tree only on the right hand branch.
		\item A node decorated by $ r $ or $\circ$ is connected to its parent in the tree only on the left hand branch.
		\item The decoration $k$ only appears at the leaves and it is always connected to a node decorated by $\circ$ which is not at the root.
		\end{enumerate}
\end{itemize}
Given a decorated tree $T$, we define $ |T| $ as
\begin{equation*}
	| T | =  2 |L_T^k| +  4 \sum_{i \in I \setminus \set{k}} |L_T^i| -  |E_T|
	\end{equation*}
	where $L_T^i$ are the leaves of $T$ decorated by $i$ and $ |L_T^i| $ is its cardinal. The term $|E_T|$ is the cardinal of $E_T$. In the sequel leaves decorated by $k$ will be interpreted as $H_0$ which is a monomial of degree $2$ whereas the other leaves will be interpreted as $H_1$ or a modification of $H_1$ with degree $4$.
	The inner nodes encode Poisson Brackets that decrease the size of the monomials by $2$, one for each incoming edge. 
	Therefore, $|T|$ computes the degree of the monomial associated with $T$.
	
 As a pictorial notation, decorated nodes will be denoted by $  	\begin{tikzpicture}[node distance=14pt, baseline=5]
	\node[leaf] (leaf) at (0,0.3) {$i $};
\end{tikzpicture}  $ with $i \in I$ and edges will be brown edges. We put more constraints on the node decorations of $T^{\mathfrak{n}}$ in the next definition.
\begin{assumption} \label{assumption_1}
 Every node $v$ in $T$ has the following form
	
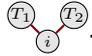
\begin{equation*}
		\begin{tikzpicture}[node distance = 14pt, baseline=10]
			\node[leaf] (inner) at (0,0.3) {$i$};
			\node[leaf] (leaf1) [above left of = inner] {$ T_1 $};
			\node[leaf] (leaf2) [above right of = inner] {$ T_2 $};
			\draw[H] (inner.north west) -- (leaf1.south east);
			\draw[H] (inner.north east) -- (leaf2.south west);
		\end{tikzpicture}.
		\end{equation*}
		Here, we have assumed that the node $v$ is decorated by $i \in I$. The trees $T_1$ and $T_2$ are subtrees of $T$ which are connected to the node $v$. By constraints (a),(b) and (c) respectively, the root of $T_2$ must be decorated by $n$ and the one of $T_1$ by $r,k$ or $\circ$.
        Then, one has 
        \begin{enumerate}[label=(\roman*)]
        	\item	If the root of $T_1$ is decorated by $\circ$, then $|T_1| \geq |T_2|$. Moreover, if $ T_1 $ is of the form
        	
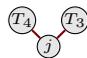
\begin{equation*}
        		\begin{tikzpicture}[node distance = 14pt, baseline=10]
        			\node[leaf] (inner) at (0,0.3) {$j $};
        			\node[leaf] (leaf1) [above left of = inner] {$ T_4 $};
        			\node[leaf] (leaf2) [above right of = inner] {$ T_3 $};
        			\draw[H] (inner.north west) -- (leaf1.south east);
        			\draw[H] (inner.north east) -- (leaf2.south west);
        		\end{tikzpicture}
        	\end{equation*}
        	then $ |T_3| \geq |T_2| $.       	
        	\item If the root of $T_1$ is decorated by $r$, then $|T_1| < |T_2|$. 
        \end{enumerate}
        \end{assumption}
  Below, we present some decorated trees satisfying the previous rules:
  \begin{equation*}
  T= 	\begin{tikzpicture}[node distance = 14pt, baseline=10]
  		\node[leaf] (inner) at (0,0.3) {$\circ $};
  		\node[leaf] (leaf1) [above left of = inner] {$ \circ $};
  		\node[leaf] (leaf2) [above right of = inner] {$ n $};
  		\draw[H] (inner.north west) -- (leaf1.south east);
  		\draw[H] (inner.north east) -- (leaf2.south west);
  	\end{tikzpicture}, \quad \bar{T} = \begin{tikzpicture}[node distance = 14pt, baseline=10]
  	\node[leaf] (inner) at (0,0.3) {$r$};
  	\node[leaf] (leaf1) [above left of = inner] {$ \circ $};
  	\node[leaf] (leaf3) [above left of = leaf1] {$ k $};
  	\node[leaf] (leaf4) [above right of = leaf1] {$ n $};
  	\node[leaf] (leaf2) [above right of = inner] {$ n $};
  	\draw[H] (inner.north west) -- (leaf1.south east);
  	\draw[H] (inner.north east) -- (leaf2.south west);
  	\draw[H] (leaf1.north east) -- (leaf4.south west);
  	\draw[H] (leaf1.north west) -- (leaf3.south east);
  	\end{tikzpicture}.
  \end{equation*}
  One has
  \begin{equs}
  	|T| & = 4 | L_T^{\circ} | + 4 | L_{T}^n | - |E_T| = 4 + 4 -2 = 6, \\ |\bar{T}| & = 2 | L_{\bar{T}}^{k} | + 4 | N_{\bar{T}}^{n} | - |E_{\bar{T}}| =  2 + 2 \times 4 - 4 = 6.
  \end{equs}
 One can easily check that the conditions on the decorations are satisfied on both decorated trees. We denote by $\CT$ the set of these decorated trees. We define recursively a map $\Pi$ that interpreted these decorated trees into iterated Poisson Lie Brackets. It is given by
 \begin{equation}
 	\label{rec_def_Pi}
 	\begin{aligned}
 	(\Pi \,	\begin{tikzpicture}[node distance=14pt, baseline=5]
 		\node[leaf] (leaf) at (0,0.3) {$k $};
 	\end{tikzpicture}) & = H_0, \quad (\Pi \,	\begin{tikzpicture}[node distance=14pt, baseline=5]
 	\node[leaf] (leaf) at (0,0.3) {$\circ $};
 	\end{tikzpicture}) = H_1, \quad 	(\Pi \,	\begin{tikzpicture}[node distance=14pt, baseline=5]
 	\node[leaf] (leaf) at (0,0.3) {$n $};
 	\end{tikzpicture}) = H_{1,\Phi}^{\text{\tiny{non-res}}}, \quad (\Pi \,	\begin{tikzpicture}[node distance=14pt, baseline=5]
 	\node[leaf] (leaf) at (0,0.3) {$r $};
 	\end{tikzpicture})  =  H_1^{\text{\tiny{res}}},
 	\\
 	(\Pi \, \begin{tikzpicture}[node distance = 14pt, baseline=10]
 		\node[leaf] (inner) at (0,0.3) {$\circ$};
 		\node[leaf] (leaf1) [above left of = inner] {$ T_1 $};
 		\node[leaf] (leaf2) [above right of = inner] {$ T_2 $};
 		\draw[H] (inner.north west) -- (leaf1.south east);
 		\draw[H] (inner.north east) -- (leaf2.south west);
 	\end{tikzpicture}) &= \set{  \Pi T_1, \Pi T_2 }, \quad (\Pi \, \begin{tikzpicture}[node distance = 14pt, baseline=10]
 	\node[leaf] (inner) at (0,0.3) {$ n $};
 	\node[leaf] (leaf1) [above left of = inner] {$ T_1 $};
 	\node[leaf] (leaf2) [above right of = inner] {$ T_2 $};
 	\draw[H] (inner.north west) -- (leaf1.south east);
 	\draw[H] (inner.north east) -- (leaf2.south west);
 	\end{tikzpicture}) = \set{  \Pi T_1, \Pi T_2 }^{\text{\tiny{non-res}}}_{\Phi},
 	\\ (\Pi \, \begin{tikzpicture}[node distance = 14pt, baseline=10]
 		\node[leaf] (inner) at (0,0.3) {$ r $};
 		\node[leaf] (leaf1) [above left of = inner] {$ T_1 $};
 		\node[leaf] (leaf2) [above right of = inner] {$ T_2 $};
 		\draw[H] (inner.north west) -- (leaf1.south east);
 		\draw[H] (inner.north east) -- (leaf2.south west);
 	\end{tikzpicture}) & = \set{  \Pi T_1, \Pi T_2 }^{\text{\tiny{res}}}.
 	\end{aligned}
 \end{equation}
 One also defines combinatorial coefficients associated with these iterated brackets.
\begin{equation}
	\label{rec_def_S}
	\begin{aligned}
		S^{j}( \,	\begin{tikzpicture}[node distance=14pt, baseline=5]
			\node[leaf] (leaf) at (0,0.3) {$k $};
		\end{tikzpicture} ) & = S^{j}(\,	\begin{tikzpicture}[node distance=14pt, baseline=5]
			\node[leaf] (leaf) at (0,0.3) {$ n $};
		\end{tikzpicture} ) = S^{j}( \,	\begin{tikzpicture}[node distance=14pt, baseline=5]
			\node[leaf] (leaf) at (0,0.3) {$\circ $};
		\end{tikzpicture} )  = S^{j}( \,	\begin{tikzpicture}[node distance=14pt, baseline=5]
			\node[leaf] (leaf) at (0,0.3) {$r $};
		\end{tikzpicture} )   = j+1, 
		\\
		S^j (\, \begin{tikzpicture}[node distance = 14pt, baseline=10]
			\node[leaf] (inner) at (0,0.3) {$\circ$};
			\node[leaf] (leaf1) [above left of = inner] {$ T_1 $};
			\node[leaf] (leaf2) [above right of = inner] {$ T_2 $};
			\draw[H] (inner.north west) -- (leaf1.south east);
			\draw[H] (inner.north east) -- (leaf2.south west);
		\end{tikzpicture}) &=  (j+1) S^{j+1}(T_1) S^{0}(T_2), \, \text{if } T_1 = \begin{tikzpicture}[node distance = 14pt, baseline=10]
			\node[leaf] (inner) at (0,0.3) {$\circ$};
			\node[leaf] (leaf1) [above left of = inner] {$ T_3 $};
			\node[leaf] (leaf2) [above right of = inner] {$ T_4 $};
			\draw[H] (inner.north west) -- (leaf1.south east);
			\draw[H] (inner.north east) -- (leaf2.south west);
		\end{tikzpicture}, \, |T_4| = |T_2|, \\ 
		& = (j+1) S^{0}(T_1) S^{0}(T_2), \, \text{otherwise},
		\\
		S^j ( \, \begin{tikzpicture}[node distance = 14pt, baseline=10]
			\node[leaf] (inner) at (0,0.3) {$ n $};
			\node[leaf] (leaf1) [above left of = inner] {$ T_1 $};
			\node[leaf] (leaf2) [above right of = inner] {$ T_2 $};
			\draw[H] (inner.north west) -- (leaf1.south east);
			\draw[H] (inner.north east) -- (leaf2.south west);
		\end{tikzpicture}) & = S^j ( \, \begin{tikzpicture}[node distance = 14pt, baseline=10]
			\node[leaf] (inner) at (0,0.3) {$ r $};
			\node[leaf] (leaf1) [above left of = inner] {$ T_1 $};
			\node[leaf] (leaf2) [above right of = inner] {$ T_2 $};
			\draw[H] (inner.north west) -- (leaf1.south east);
			\draw[H] (inner.north east) -- (leaf2.south west);
		\end{tikzpicture}) =  (j+1) \, S^{0}(T_1) S^{0}(T_2)
	\end{aligned}
\end{equation}
 and we set
 \begin{equation*}
 	S(T) = S^{0}(T).
 \end{equation*}
        As an example, one has
        \begin{equation*}
         S( 	\begin{tikzpicture}[node distance = 14pt, baseline=10]
        	\node[leaf] (inner) at (0,0.3) {$\circ $};
        	\node[leaf] (leaf1) [above left of = inner] {$ \circ $};
        	\node[leaf] (leaf2) [above right of = inner] {$ n $};
        	\draw[H] (inner.north west) -- (leaf1.south east);
        	\draw[H] (inner.north east) -- (leaf2.south west);
        \end{tikzpicture}) = 1, \quad S( \begin{tikzpicture}[node distance = 14pt, baseline=10]
        	\node[leaf] (inner) at (0,0.3) {$r$};
        	\node[leaf] (leaf1) [above left of = inner] {$ \circ $};
        	\node[leaf] (leaf3) [above left of = leaf1] {$ k $};
        	\node[leaf] (leaf4) [above right of = leaf1] {$ n $};
        	\node[leaf] (leaf2) [above right of = inner] {$ n $};
        	\draw[H] (inner.north west) -- (leaf1.south east);
        	\draw[H] (inner.north east) -- (leaf2.south west);
        	\draw[H] (leaf1.north east) -- (leaf4.south west);
        	\draw[H] (leaf1.north west) -- (leaf3.south east);
        \end{tikzpicture}) = 2 .
        \end{equation*}
        For the second example, one is able to find a path of length $3$ with nodes successively  decorated by $r,\circ,k$. This gives $2! =2$. The recursive definition \eqref{rec_def_S} produces factorial coefficients depending on the path inside the tree containing nodes decorated by $\circ$. If we consider a third tree $ T_3 $ for which $ S(T_3)=3! $ then we have 
        \begin{equation*}
        	S( \begin{tikzpicture}[node distance = 14pt, baseline=10]
        		\node[leaf] (inner) at (0,0.3) {$r$};
        		\node[leaf] (leaf1) [above left of = inner] {$ \circ $};
        		\node[leaf] (leaf3) [above left of = leaf1] {$ k $};
        		\node[leaf] (leaf4) [above right of = leaf1] {$ n $};
        		\node[leaf] (leaf2) [above right of = inner] {$ T_3 $};
        		\draw[H] (inner.north west) -- (leaf1.south east);
        		\draw[H] (inner.north east) -- (leaf2.south west);
        		\draw[H] (leaf1.north east) -- (leaf4.south west);
        		\draw[H] (leaf1.north west) -- (leaf3.south east);
        	\end{tikzpicture}) =2!3!
        \end{equation*} 
    	because the two paths are separated by the node decorated by $ n $ at the root of $ T_3 $ which is we know is $ n $ due to constraint (a).  In terms of the iterated brackets these paths coincide with the repetitions of particular symplectic transforms as described in definition \ref{def::truncation-term}. The following proposition provides a way to extract the Taylor coefficient of the corresponding term in the normal form from the tree.
        \begin{proposition} \label{prop_iterated_bracket} For $T_1,...,T_p$ decorated trees, one has
        	\begin{equation*}
        		\begin{aligned}
        			\frac{1}{p!}	\set{ \set{\frac{(\Pi T_1)}{S(T_1)}, \frac{(\Pi T_{2})}{S(T_2)}} \cdots, \frac{(\Pi T_p)}{S(T_p)}  } = \frac{(\Pi T)}{S(T)}, \quad T = \begin{tikzpicture}[node distance = 14pt, baseline=10]
        				\node[leaf] (inner) at (0,0.3) {$\circ$};
        				\node[leaf] (leaf1) [above left of = inner] {$ \tiny{\cdots} $};
        				\node[leaf] (leaf3) [above left of = leaf1] {$ \circ $};
        				\node[leaf] (leaf5) [above left of = leaf3] {$ T_1 $};
        				\node[leaf] (leaf6) [above right of = leaf3] {$ T_2 $};
        				\node[leaf] (leaf2) [above right of = inner] {$ T_p $};
        				\draw[H] (inner.north west) -- (leaf1.south east);
        				\draw[H] (inner.north east) -- (leaf2.south west);
        				\draw[H] (leaf1.north west) -- (leaf3.south east);
        				\draw[H] (leaf3.north east) -- (leaf6.south west);
        				\draw[H] (leaf3.north west) -- (leaf5.south east);
        			\end{tikzpicture}.
        		\end{aligned}
        	\end{equation*}
        	and
        	\begin{equation*}
        		\begin{aligned}
        			S(T) = p! \prod_{i=1}^{p} S(T_i).
        		\end{aligned}
        	\end{equation*}
        \end{proposition}
        \begin{proof} We show by induction on $p \in \mathbb{N}^*$ that
        	\begin{equation*}
        		\set{ \set{(\Pi T_1), (\Pi T_{2})} \cdots, (\Pi T_p)  } = (\Pi T).
        	\end{equation*}
        	The property is trivially true for $p=1$. We now assume this property true for $p \in \mathbb{N}^*$. One has
        	\begin{equation}
        		\label{ident_n}
        		\begin{aligned}
        			\set{ \set{(\Pi T_1), (\Pi T_{2})} \cdots, (\Pi T_{p+1})  } & = 		\set{ (\Pi T), (\Pi T_{p+1})} 
        			\\ & = 	(\Pi \hat{T}) 
        		\end{aligned}
        	\end{equation}
        	with 
        	\begin{equation*}
        		T = \begin{tikzpicture}[node distance = 14pt, baseline=10]
        			\node[leaf] (inner) at (0,0.3) {$\circ$};
        			\node[leaf] (leaf1) [above left of = inner] {$ \tiny{\cdots} $};
        			\node[leaf] (leaf3) [above left of = leaf1] {$ \circ $};
        			\node[leaf] (leaf5) [above left of = leaf3] {$ T_1 $};
        			\node[leaf] (leaf6) [above right of = leaf3] {$ T_2 $};
        			\node[leaf] (leaf2) [above right of = inner] {$ T_p $};
        			\draw[H] (inner.north west) -- (leaf1.south east);
        			\draw[H] (inner.north east) -- (leaf2.south west);
        			\draw[H] (leaf1.north west) -- (leaf3.south east);
        			\draw[H] (leaf3.north east) -- (leaf6.south west);
        			\draw[H] (leaf3.north west) -- (leaf5.south east);
        		\end{tikzpicture}
        		, \quad 
        		\hat{T} = \begin{tikzpicture}[node distance = 14pt, baseline=10]
        			\node[leaf] (inner) at (0,0.3) {$\circ$};
        			\node[leaf] (leaf1) [above left of = inner] {$ \tiny{\cdots} $};
        			\node[leaf] (leaf3) [above left of = leaf1] {$ \circ $};
        			\node[leaf] (leaf5) [above left of = leaf3] {$ T_1 $};
        			\node[leaf] (leaf6) [above right of = leaf3] {$ T_2 $};
        			\node[leaf] (leaf2) [above right of = inner] {$ T_{\bar{p}} $};
        			\draw[H] (inner.north west) -- (leaf1.south east);
        			\draw[H] (inner.north east) -- (leaf2.south west);
        			\draw[H] (leaf1.north west) -- (leaf3.south east);
        			\draw[H] (leaf3.north east) -- (leaf6.south west);
        			\draw[H] (leaf3.north west) -- (leaf5.south east);
        		\end{tikzpicture},\quad \bar{p} = p+1.
        	\end{equation*}
        	For the first line of \eqref{ident_n}, we have applied the induction hypothesis. The second line is just an application of the recursive definition  of $\Pi$ (see 
        	\eqref{rec_def_Pi}). We show by induction on $p \in \mathbb{N}^*$ that
        	\begin{equation*}
        		S^j(T) =  \prod_{i=1}^{p} (j+i) S(T_{p+1-i}).
        	\end{equation*}
        	This is true for $p=1$. Let assume the property true for $p \in \mathbb{N}^*$. One has from the definition of $S^j(\hat{T})$ when $|T_p| = |T_{p+1}|$
        	\begin{equation*}
        		S^j(\hat{T}) = (j+1) S^{j+1}(T) S(T_{p+1})
        	\end{equation*}
        	Using the induction hypothesis, one has
        	\begin{equation*}
        		\begin{aligned}
        			S^{j+1}(T) & = \prod_{i=1}^{p} (j+i+1) S(T_{p+1-i})
        			 = \prod_{i=2}^{p+1} (j+i) S(T_{p+2-i})
        		\end{aligned}
        	\end{equation*}
        	with the change of indice $i \rightarrow i+1$.
        	Then, one has
        	\begin{equation*}
        		\begin{aligned}
        			S^j(\hat{T}) & = (j+1)  S(T_{p+1}) \prod_{i=2}^{p+1} (j+i) S(T_{p+2-i})
        			  = \prod_{i=1}^{p+1} (j+i) S(T_{p+2-i})
        		\end{aligned}
        	\end{equation*}
        	which allows us to conclude.
        \end{proof}
        
        Before rewriting the Birkhoff normal form with this decorated trees formalism, we introduce some useful subsets of decorated trees in the next definition
        \begin{definition} \label{spaces_trees}
        	One set the following subsets of decorated trees for $m \in \mathbb{N}^*$:
        	\begin{itemize}
        		\item $\mathcal{T}^{\! < m}_{r}$: decorated trees $T$ with the root decorated by $r$ and such that $|T| < 2 m$.
 \item $\mathcal{T}^{\! m}_{\circ}$: decorated trees $T$ with the root decorated by $\circ$ and such that $|T| = 2 m$.
 \item $\mathcal{T}^{\! m}_{n}$: decorated trees $T$ with the root decorated by $n$ and such that $|T| = 2 m$.
        			\item $\mathcal{T}^{\!  m,\ell}_{\circ}$: decorated trees $T$ with the root decorated by $\circ$ and such that $ 2m <|T| \leq 2 \ell$. Moreover, we assume that they do not contain any node of the form:
        			
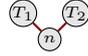
\begin{equation*}
        				\begin{tikzpicture}[node distance = 14pt, baseline=10]
        					\node[leaf] (inner) at (0,0.3) {$n$};
        					\node[leaf] (leaf1) [above left of = inner] {$ T_1 $};
        					\node[leaf] (leaf2) [above right of = inner] {$ T_2 $};
        					\draw[H] (inner.north west) -- (leaf1.south east);
        					\draw[H] (inner.north east) -- (leaf2.south west);
        				\end{tikzpicture}
        			\end{equation*}
        			with $| T_1 | + |T_2| -2  \geq 2 m$.
        	\end{itemize}
        	\end{definition}
     We rewrite the previous computation of the Birkhoff normal form using these spaces. One has
     	\begin{equation*}
     		\begin{aligned}
     			\msH_1^3=(H\circ F_1)^3  & = H_0 + H_1^{\text{\tiny{res}}} + \set{H_1, H_{1,\Phi}^{\text{\tiny{non-res}}}} + \frac12\set{\set{H_0,H_{1,\Phi}^{\text{\tiny{non-res}}}}, H_{1,\Phi}^{\text{\tiny{non-res}}}}  
     			\\ & = (\Pi \, \begin{tikzpicture}[node distance=14pt, baseline=5]
     				\node[leaf] (leaf) at (0,0.3) {$ k $};
     			\end{tikzpicture} ) +	 ( \Pi \, \begin{tikzpicture}[node distance=14pt, baseline=5]
     				\node[leaf] (leaf) at (0,0.3) {$r$};
     			\end{tikzpicture} ) + (\Pi \, \begin{tikzpicture}[node distance = 14pt, baseline=10]
     				\node[leaf] (inner) at (0,0.3) {$\circ $};
     				\node[leaf] (leaf1) [above left of = inner] {$ \circ $};
     				\node[leaf] (leaf2) [above right of = inner] {$ n $};
     				\draw[H] (inner.north west) -- (leaf1.south east);
     				\draw[H] (inner.north east) -- (leaf2.south west);
     			\end{tikzpicture} ) + \frac12 \, (\Pi \, \begin{tikzpicture}[node distance = 14pt, baseline=10]
     				\node[leaf] (inner) at (0,0.3) {$\circ $};
     				\node[leaf] (leaf1) [above left of = inner] {$ \circ $};
     				\node[leaf] (leaf3) [above left of = leaf1] {$ k $};
     				\node[leaf] (leaf4) [above right of = leaf1] {$ n $};
     				\node[leaf] (leaf2) [above right of = inner] {$ n $};
     				\draw[H] (inner.north west) -- (leaf1.south east);
     				\draw[H] (inner.north east) -- (leaf2.south west);
     				\draw[H] (leaf1.north east) -- (leaf4.south west);
     				\draw[H] (leaf1.north west) -- (leaf3.south east);
     			\end{tikzpicture})
     			\\ & =  (\Pi \, \begin{tikzpicture}[node distance=14pt, baseline=5]
     				\node[leaf] (leaf) at (0,0.3) {$ k $};
     			\end{tikzpicture} ) +  \sum_{T \in \CT_r^{< 3}}  \frac{(\Pi T)}{S(T)} + 
     			\sum_{T \in \CT_{\circ}^{3}}  \frac{(\Pi T)}{S(T)}, 
     		\end{aligned}
     	\end{equation*}
     	where one has
     	\begin{equation*}
     	\CT_r^{< 3} = \set{ \begin{tikzpicture}[node distance=14pt, baseline=5]
     			\node[leaf] (leaf) at (0,0.3) {$ r $};
     	\end{tikzpicture} }, \quad \CT_{\circ}^{3} = \set{ \begin{tikzpicture}[node distance = 14pt, baseline=10]
     	\node[leaf] (inner) at (0,0.3) {$\circ $};
     	\node[leaf] (leaf1) [above left of = inner] {$ \circ $};
     	\node[leaf] (leaf2) [above right of = inner] {$ n $};
     	\draw[H] (inner.north west) -- (leaf1.south east);
     	\draw[H] (inner.north east) -- (leaf2.south west);
     	\end{tikzpicture}, \, \begin{tikzpicture}[node distance = 14pt, baseline=10]
     	\node[leaf] (inner) at (0,0.3) {$\circ $};
     	\node[leaf] (leaf1) [above left of = inner] {$ \circ $};
     	\node[leaf] (leaf3) [above left of = leaf1] {$ k $};
     	\node[leaf] (leaf4) [above right of = leaf1] {$ n $};
     	\node[leaf] (leaf2) [above right of = inner] {$ n $};
     	\draw[H] (inner.north west) -- (leaf1.south east);
     	\draw[H] (inner.north east) -- (leaf2.south west);
     	\draw[H] (leaf1.north east) -- (leaf4.south west);
     	\draw[H] (leaf1.north west) -- (leaf3.south east);
     \end{tikzpicture} }, \quad \mathcal{T}^{\! 3,4}_{\circ} = \set{\, \begin{tikzpicture}[node distance = 14pt, baseline=10]
     \node[leaf] (inner) at (0,0.3) {$\circ$};
     \node[leaf] (leaf1) [above left of = inner] {$ \circ $};
     \node[leaf] (leaf3) [above left of = leaf1] {$ \circ $};
     \node[leaf] (leaf4) [above right of = leaf1] {$ n $};
     \node[leaf] (leaf2) [above right of = inner] {$ n $};
     \draw[H] (inner.north west) -- (leaf1.south east);
     \draw[H] (inner.north east) -- (leaf2.south west);
     \draw[H] (leaf1.north east) -- (leaf4.south west);
     \draw[H] (leaf1.north west) -- (leaf3.south east);
     \end{tikzpicture}, \, \begin{tikzpicture}[node distance = 14pt, baseline=10]
     \node[leaf] (inner) at (0,0.3) {$\circ $};
     \node[leaf] (leaf1) [above left of = inner] {$ \circ $};
     \node[leaf] (leaf3) [above left of = leaf1] {$ \circ $};
      \node[leaf] (leaf5) [above right of = leaf3] {$ n $};
      \node[leaf] (leaf6) [above left of = leaf3] {$ k $};
     \node[leaf] (leaf4) [above right of = leaf1] {$ n $};
     \node[leaf] (leaf2) [above right of = inner] {$ n $};
     \draw[H] (inner.north west) -- (leaf1.south east);
     \draw[H] (inner.north east) -- (leaf2.south west);
      \draw[H] (leaf3.north east) -- (leaf5.south west);
      \draw[H] (leaf3.north west) -- (leaf6.south east);
     \draw[H] (leaf1.north east) -- (leaf4.south west);
     \draw[H] (leaf1.north west) -- (leaf3.south east);
 \end{tikzpicture}  }.
     	\end{equation*}
     	We notice that
     	\begin{equation*}
     		F_1 = H_{1,\Phi}^{\text{\tiny{non-res}}} = \sum_{T \in \CT_n^{2}} \frac{(\Pi T)}{S(T)}, \quad \CT_n^{2} = \set{\begin{tikzpicture}[node distance=14pt, baseline=5]
     				\node[leaf] (leaf) at (0,0.3) {$ n $};
     		\end{tikzpicture}}.
     	\end{equation*}
   We now consider the second step of the Birkhoff normal form when one uses a second symplectic transform.
   One has
   \begin{equation*}
   	\begin{aligned}
   	F_2 & = 	 \set{H_1, H_{1,\Phi}^{\text{\tiny{non-res}}} }^{\text{\tiny{non-res}}}_{\Phi}  + \frac12\set{\set{H_0,H_{1,\Phi}^{\text{\tiny{non-res}}} }, H_{1,\Phi}^{\text{\tiny{non-res}}} }^{\text{\tiny{non-res}}}_{\Phi}
   	\\ & = (\Pi \, \begin{tikzpicture}[node distance = 14pt, baseline=10]
   		\node[leaf] (inner) at (0,0.3) {$ n $};
   		\node[leaf] (leaf1) [above left of = inner] {$ \circ $};
   		\node[leaf] (leaf2) [above right of = inner] {$ n $};
   		\draw[H] (inner.north west) -- (leaf1.south east);
   		\draw[H] (inner.north east) -- (leaf2.south west);
   	\end{tikzpicture} )+ \frac12 (\Pi \, \begin{tikzpicture}[node distance = 14pt, baseline=10]
   		\node[leaf] (inner) at (0,0.3) {$n $};
   		\node[leaf] (leaf1) [above left of = inner] {$ \circ $};
   		\node[leaf] (leaf3) [above left of = leaf1] {$ k $};
   		\node[leaf] (leaf4) [above right of = leaf1] {$ n $};
   		\node[leaf] (leaf2) [above right of = inner] {$ n $};
   		\draw[H] (inner.north west) -- (leaf1.south east);
   		\draw[H] (inner.north east) -- (leaf2.south west);
   		\draw[H] (leaf1.north east) -- (leaf4.south west);
   		\draw[H] (leaf1.north west) -- (leaf3.south east);
   	\end{tikzpicture}) = \sum_{T \in \CT_n^{3}} \frac{(\Pi T)}{S(T)}
   	\end{aligned}
   \end{equation*}
   where 
   \begin{equation*}
   	\CT_n^{3} = \set{\begin{tikzpicture}[node distance = 14pt, baseline=10]
   			\node[leaf] (inner) at (0,0.3) {$n $};
   			\node[leaf] (leaf1) [above left of = inner] {$ \circ $};
   			\node[leaf] (leaf2) [above right of = inner] {$ n $};
   			\draw[H] (inner.north west) -- (leaf1.south east);
   			\draw[H] (inner.north east) -- (leaf2.south west);
   		\end{tikzpicture}, \, \begin{tikzpicture}[node distance = 14pt, baseline=10]
   			\node[leaf] (inner) at (0,0.3) {$ n $};
   			\node[leaf] (leaf1) [above left of = inner] {$ \circ $};
   			\node[leaf] (leaf3) [above left of = leaf1] {$ k $};
   			\node[leaf] (leaf4) [above right of = leaf1] {$ n $};
   			\node[leaf] (leaf2) [above right of = inner] {$ n $};
   			\draw[H] (inner.north west) -- (leaf1.south east);
   			\draw[H] (inner.north east) -- (leaf2.south west);
   			\draw[H] (leaf1.north east) -- (leaf4.south west);
   			\draw[H] (leaf1.north west) -- (leaf3.south east);
   	\end{tikzpicture}}.
   \end{equation*}
   Then, 
   \begin{align*}
   	\msH_2^4 &= (\msH_1\circ F_2)^4 \\
   	&= H_0 + H_1^{\text{\tiny{res}}} + \set{H_1, H_{1,\Phi}^{\text{\tiny{non-res}}}}^{\text{\tiny{res}}} + \frac12\set{\set{H_0,H_{1,\Phi}^{\text{\tiny{non-res}}}} , H_{1,\Phi}^{\text{\tiny{non-res}}}}^{\text{\tiny{res}}} \\ & + \frac12 \set{\set{H_1, H_{1,\Phi}^{\text{\tiny{non-res}}}},H_{1,\Phi}^{\text{\tiny{non-res}}}} + \set{H_1^\text{\tiny{res}},\set{H_1, H_{1,\Phi}^{\text{\tiny{non-res}}} }^{\text{\tiny{non-res}}}_{\Phi}} 
   	\\ & + \frac12 \set{H_1^\text{\tiny{res}}, \set{\set{H_0,H_{1,\Phi}^{\text{\tiny{non-res}}} }, H_{1,\Phi}^{\text{\tiny{non-res}}} }^{\text{\tiny{non-res}}}_{\Phi}} 
   	\\ &  = (\Pi \, \begin{tikzpicture}[node distance=14pt, baseline=5]
   		\node[leaf] (leaf) at (0,0.3) {$k $};
   	\end{tikzpicture}) +	(\Pi \, \begin{tikzpicture}[node distance=14pt, baseline=5]
   		\node[leaf] (leaf) at (0,0.3) {$r$};
   	\end{tikzpicture}) + (  \Pi \,\begin{tikzpicture}[node distance = 14pt, baseline=10]
   		\node[leaf] (inner) at (0,0.3) {$r$};
   		\node[leaf] (leaf1) [above left of = inner] {$ \circ $};
   		\node[leaf] (leaf2) [above right of = inner] {$ n $};
   		\draw[H] (inner.north west) -- (leaf1.south east);
   		\draw[H] (inner.north east) -- (leaf2.south west);
   	\end{tikzpicture} )+ \frac12 \, (\Pi \begin{tikzpicture}[node distance = 14pt, baseline=10]
   		\node[leaf] (inner) at (0,0.3) {$r$};
   		\node[leaf] (leaf1) [above left of = inner] {$ \circ $};
   		\node[leaf] (leaf3) [above left of = leaf1] {$ k $};
   		\node[leaf] (leaf4) [above right of = leaf1] {$ n $};
   		\node[leaf] (leaf2) [above right of = inner] {$ n $};
   		\draw[H] (inner.north west) -- (leaf1.south east);
   		\draw[H] (inner.north east) -- (leaf2.south west);
   		\draw[H] (leaf1.north east) -- (leaf4.south west);
   		\draw[H] (leaf1.north west) -- (leaf3.south east);
   	\end{tikzpicture}) + \frac12 \, ( \Pi \,
  \begin{tikzpicture}[node distance = 14pt, baseline=10]
   		\node[leaf] (inner) at (0,0.3) {$\circ$};
   		\node[leaf] (leaf1) [above left of = inner] {$ \circ $};
   		\node[leaf] (leaf3) [above left of = leaf1] {$ \circ $};
   		\node[leaf] (leaf4) [above right of = leaf1] {$ n $};
   		\node[leaf] (leaf2) [above right of = inner] {$ n $};
   		\draw[H] (inner.north west) -- (leaf1.south east);
   		\draw[H] (inner.north east) -- (leaf2.south west);
   		\draw[H] (leaf1.north east) -- (leaf4.south west);
   		\draw[H] (leaf1.north west) -- (leaf3.south east);
   	\end{tikzpicture})
   \\ & 	+ ( \Pi \, \begin{tikzpicture}[node distance = 14pt, baseline=10]
   		\node[leaf] (inner) at (0,0.3) {$\circ$};
   		\node[leaf] (leaf1) [above left of = inner] {$ r $};
   		\node[leaf] (leaf2) [above right of = inner] {$ n $};
   		\node[leaf] (leaf3) [above right of = leaf2] {$ n $};
   		\node[leaf] (leaf4) [above left of = leaf2] {$ \circ $};
   		\draw[H] (inner.north west) -- (leaf1.south east);
   		\draw[H] (leaf2.north west) -- (leaf4.south east);
   		\draw[H] (leaf2.north east) -- (leaf3.south west);
   		\draw[H] (inner.north east) -- (leaf2.south west);
   	\end{tikzpicture})
   	+ \frac12 \, (\Pi \, \begin{tikzpicture}[node distance = 14pt, baseline=10]
   		\node[leaf] (inner) at (0,0.3) {$\circ$};
   		\node[leaf] (leaf1) [above left of = inner] {$ r $};
   		\node[leaf] (leaf2) [above right of = inner] {$ n $};
   		\node[leaf] (leaf3) [above right of = leaf2] {$ n $};
   		\node[leaf] (leaf4) [above left of = leaf2] {$ \circ $};
   		\node[leaf] (leaf5) [above left of = leaf4] {$ k $};
   		\node[leaf] (leaf6) [above right of = leaf4] {$ n $};
   		\draw[H] (inner.north west) -- (leaf1.south east);
   		\draw[H] (leaf4.north west) -- (leaf5.south east);
   		\draw[H] (leaf4.north east) -- (leaf6.south west);
   		\draw[H] (leaf2.north west) -- (leaf4.south east);
   		\draw[H] (leaf2.north east) -- (leaf3.south west);
   		\draw[H] (inner.north east) -- (leaf2.south west);
   	\end{tikzpicture}) + \frac{1}{6} \, ( \Pi \,
   	\begin{tikzpicture}[node distance = 14pt, baseline=10]
   	\node[leaf] (inner) at (0,0.3) {$\circ$};
   	\node[leaf] (leaf1) [above left of = inner] {$ \circ $};
   	\node[leaf] (leaf3) [above left of = leaf1] {$ \circ $};
   	\node[leaf] (leaf5) [above left of = leaf3] {$ k $};
   	\node[leaf] (leaf6) [above right of = leaf3] {$ n $};
   	\node[leaf] (leaf4) [above right of = leaf1] {$ n $};
   	\node[leaf] (leaf2) [above right of = inner] {$ n $};
   	\draw[H] (inner.north west) -- (leaf1.south east);
   	\draw[H] (inner.north east) -- (leaf2.south west);
   	\draw[H] (leaf1.north east) -- (leaf4.south west);
   	\draw[H] (leaf1.north west) -- (leaf3.south east);
   	\draw[H] (leaf3.north east) -- (leaf6.south west);
   	\draw[H] (leaf3.north west) -- (leaf5.south east);
   	\end{tikzpicture})
   	\\ & =  (\Pi \, \begin{tikzpicture}[node distance=14pt, baseline=5]
   		\node[leaf] (leaf) at (0,0.3) {$ k $};
   	\end{tikzpicture} ) +  \sum_{T \in \CT_r^{< 4}}  \frac{(\Pi T)}{S(T)} + 
   	\sum_{T \in \CT_{\circ}^{4}}  \frac{(\Pi T)}{S(T)}, 
   \end{align*}
   where
   \begin{equation*}
   	\begin{aligned}
   	\CT_r^{< 4} & = \set{ \begin{tikzpicture}[node distance=14pt, baseline=5]
   			\node[leaf] (leaf) at (0,0.3) {$r$};
   		\end{tikzpicture}, \, \begin{tikzpicture}[node distance = 14pt, baseline=10]
   			\node[leaf] (inner) at (0,0.3) {$r$};
   			\node[leaf] (leaf1) [above left of = inner] {$ \circ $};
   			\node[leaf] (leaf2) [above right of = inner] {$ n $};
   			\draw[H] (inner.north west) -- (leaf1.south east);
   			\draw[H] (inner.north east) -- (leaf2.south west);
   		\end{tikzpicture}, \,  \begin{tikzpicture}[node distance = 14pt, baseline=10]
   			\node[leaf] (inner) at (0,0.3) {$r$};
   			\node[leaf] (leaf1) [above left of = inner] {$ \circ $};
   			\node[leaf] (leaf3) [above left of = leaf1] {$ k $};
   			\node[leaf] (leaf4) [above right of = leaf1] {$ n $};
   			\node[leaf] (leaf2) [above right of = inner] {$ n $};
   			\draw[H] (inner.north west) -- (leaf1.south east);
   			\draw[H] (inner.north east) -- (leaf2.south west);
   			\draw[H] (leaf1.north east) -- (leaf4.south west);
   			\draw[H] (leaf1.north west) -- (leaf3.south east);
   		\end{tikzpicture} \, }, \\ \CT_{\circ}^{4} & = \set{    \begin{tikzpicture}[node distance = 14pt, baseline=10]
   		\node[leaf] (inner) at (0,0.3) {$\circ$};
   		\node[leaf] (leaf1) [above left of = inner] {$ \circ $};
   		\node[leaf] (leaf3) [above left of = leaf1] {$ \circ $};
   		\node[leaf] (leaf4) [above right of = leaf1] {$ n $};
   		\node[leaf] (leaf2) [above right of = inner] {$ n $};
   		\draw[H] (inner.north west) -- (leaf1.south east);
   		\draw[H] (inner.north east) -- (leaf2.south west);
   		\draw[H] (leaf1.north east) -- (leaf4.south west);
   		\draw[H] (leaf1.north west) -- (leaf3.south east);
   		\end{tikzpicture},
   \, \begin{tikzpicture}[node distance = 14pt, baseline=10]
   		\node[leaf] (inner) at (0,0.3) {$\circ$};
   		\node[leaf] (leaf1) [above left of = inner] {$ r $};
   		\node[leaf] (leaf2) [above right of = inner] {$ n $};
   		\node[leaf] (leaf3) [above right of = leaf2] {$ n $};
   		\node[leaf] (leaf4) [above left of = leaf2] {$ \circ $};
   		\draw[H] (inner.north west) -- (leaf1.south east);
   		\draw[H] (leaf2.north west) -- (leaf4.south east);
   		\draw[H] (leaf2.north east) -- (leaf3.south west);
   		\draw[H] (inner.north east) -- (leaf2.south west);
   		\end{tikzpicture},
   		\,  \begin{tikzpicture}[node distance = 14pt, baseline=10]
   		\node[leaf] (inner) at (0,0.3) {$\circ$};
   		\node[leaf] (leaf1) [above left of = inner] {$ r $};
   		\node[leaf] (leaf2) [above right of = inner] {$ n $};
   		\node[leaf] (leaf3) [above right of = leaf2] {$ n $};
   		\node[leaf] (leaf4) [above left of = leaf2] {$ \circ $};
   		\node[leaf] (leaf5) [above left of = leaf4] {$ k $};
   		\node[leaf] (leaf6) [above right of = leaf4] {$ n $};
   		\draw[H] (inner.north west) -- (leaf1.south east);
   		\draw[H] (leaf4.north west) -- (leaf5.south east);
   		\draw[H] (leaf4.north east) -- (leaf6.south west);
   		\draw[H] (leaf2.north west) -- (leaf4.south east);
   		\draw[H] (leaf2.north east) -- (leaf3.south west);
   		\draw[H] (inner.north east) -- (leaf2.south west);
   		\end{tikzpicture}, \, \begin{tikzpicture}[node distance = 14pt, baseline=10]
   		\node[leaf] (inner) at (0,0.3) {$\circ$};
   		\node[leaf] (leaf1) [above left of = inner] {$ \circ $};
   		\node[leaf] (leaf3) [above left of = leaf1] {$ \circ $};
   		\node[leaf] (leaf5) [above left of = leaf3] {$ k $};
   		\node[leaf] (leaf6) [above right of = leaf3] {$ n $};
   		\node[leaf] (leaf4) [above right of = leaf1] {$ n $};
   		\node[leaf] (leaf2) [above right of = inner] {$ n $};
   		\draw[H] (inner.north west) -- (leaf1.south east);
   		\draw[H] (inner.north east) -- (leaf2.south west);
   		\draw[H] (leaf1.north east) -- (leaf4.south west);
   		\draw[H] (leaf1.north west) -- (leaf3.south east);
   		\draw[H] (leaf3.north east) -- (leaf6.south west);
   		\draw[H] (leaf3.north west) -- (leaf5.south east);
   	\end{tikzpicture} \,}.
   	\end{aligned}
   \end{equation*}
   One can make a few remarks on these computations:
   \begin{itemize} 
   \item 	Moving from $\CT_r^{< 3}$ to $ \CT_r^{< 4} $, we have taken the decorated trees from $\CT_{\circ}^3$, changed the decorations of their root into $r$ and added them to those of $\CT_r^{< 3}$.
   \item The set $ \CT_{\circ}^4 $ is formed of two types of decorated trees: Those from $ \CT_{\circ}^{3,4} $ and those with the following form:
   \begin{equation*}
   	\begin{tikzpicture}[node distance = 14pt, baseline=10]
   		\node[leaf] (inner) at (0,0.3) {$ \circ $};
   		\node[leaf] (leaf1) [above left of = inner] {$ T_1 $};
   		\node[leaf] (leaf2) [above right of = inner] {$ T_2 $};
   		\draw[H] (inner.north west) -- (leaf1.south east);
   		\draw[H] (inner.north east) -- (leaf2.south west);
   	\end{tikzpicture}
   	\end{equation*}
   	with $T_1 \in \CT_r^{< 3} $ and $ T_2 \in \CT_n^{3}$.
   	\end{itemize}
   These two remarks and the collection of sets $ \CT_{\circ}^{\ell, m} $ allow us to perform the proof of an explicit description of the Birkhoff normal form with decorated trees. This is the main result of the present paper and it is stated in the next Theorem:
\begin{theorem} \label{main_theorem}
	Let $m, \ell \in \mathbb{N}^{*}$ with $ m < \ell$. One has
	\begin{equation*}
		\begin{aligned}
			\msH_m^{\ell}=(H\circ F_1 \circ \cdots \circ F_m)^{\ell} & =  H_0 + \sum_{T \in \, \mathcal{T}^{\! < m + 2}_r}  \frac{(\Pi T)}{S(T)} +  \sum_{T \in \, \mathcal{T}^{\! m + 2}_{\circ}}  \frac{(\Pi T)}{S(T)} \\ & + \sum_{T \in \, \mathcal{T}_{\circ}^{m + 2,\ell}}   \frac{(\Pi T)}{S(T)}
			\end{aligned}
	\end{equation*}
	Moreover, one has
	\begin{equation*}
		F_i =  \sum_{T \in \, \mathcal{T}^{\! i+1}_{n}} \frac{(\Pi T)}{S(T)}, \quad \set{H_0, F_i} =  - \sum_{T \in \, \mathcal{T}^{\! i+1}_{\circ}} \frac{(\Pi T)^{\text{\tiny{non-res}}}}{S(T)}.
	\end{equation*}
	\end{theorem}
	\begin{proof}
		We proceed by induction on $m$. We suppose the property proved for $m \in \mathbb{N}^*$. Let start with $m=1$. In that case, one has
		\begin{equation*}
			\begin{aligned}
		\msH_1^{\ell}=(H\circ F_1)^{\ell}  & = H_0+ H_1 +  \set{H_0, F_1} + \set{H_1, F_1} + \frac12 \set{\set{H_0, F_1}, F_1} \\ &  + \frac12 \set{\set{H_1, F_1}, F_1} +  \sum_{n=3}^{\ell}\frac{\set{H,F_1}^{n}}{n!}.
		\end{aligned}
		\end{equation*}
		By definition, one has
		\begin{equation*}
			\set{H_0, F_1} = - H_1^{\text{\tiny{non-res}}},
			\end{equation*}
			which gives
			\begin{equation*}
				\begin{aligned}
			\msH_1^{\ell}= &   H_0 + \sum_{T \in \, \mathcal{T}^{\! < 3}_r}  \frac{(\Pi T)}{S(T)} +  \sum_{T \in \, \mathcal{T}^{\! 3}_{\circ}}  \frac{(\Pi T)}{S(T)}  + \sum_{T \in \, \mathcal{T}_{\circ}^{3,\ell}}   \frac{(\Pi T)}{S(T)},
			\end{aligned}
			\end{equation*}
			where
			\begin{equation*}
				\begin{aligned}
				\sum_{T \in \, \mathcal{T}^{\! 3}_{\circ}}  \frac{(\Pi T)}{S(T)} & = H_1^{\text{\tiny{res}}}, \quad \sum_{T \in \, \mathcal{T}^{\! 3}_{\circ}}  \frac{(\Pi T)}{S(T)} = \set{H_1, F_1} + \frac12 \set{\set{H_0, F_1}, F_1}, \\ \sum_{T \in \, \mathcal{T}_{\circ}^{3,\ell}}   \frac{(\Pi T)}{S(T)} & =  \frac12 \set{\set{H_1, F_1}, F_1} +  \sum_{n=3}^{\ell}\frac{\set{H,F_1}^{n}}{n!}.
				\end{aligned}
				\end{equation*}
				The first two identities are immediate. For the third one, we need to show that $\mathcal{T}_{\circ}^{3,\ell}$ is the good set of decorated trees for encoding all the iterated brackets.
			One first notices that by definition, one cannot have	
		
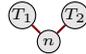
\begin{equation*}
			\begin{tikzpicture}[node distance = 14pt, baseline=10]
				\node[leaf] (inner) at (0,0.3) {$n$};
				\node[leaf] (leaf1) [above left of = inner] {$ T_1 $};
				\node[leaf] (leaf2) [above right of = inner] {$ T_2 $};
				\draw[H] (inner.north west) -- (leaf1.south east);
				\draw[H] (inner.north east) -- (leaf2.south west);
			\end{tikzpicture}
		\end{equation*}
		with $| T_1 | + |T_2| -2  \geq 2 m  =6$. This implies that the nodes decorated by $n$ must be leaves. Then, the inner nodes must be of the form
		
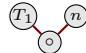
\begin{equation*}
			\begin{tikzpicture}[node distance = 14pt, baseline=10]
				\node[leaf] (inner) at (0,0.3) {$\circ$};
				\node[leaf] (leaf1) [above left of = inner] {$ T_1 $};
				\node[leaf] (leaf2) [above right of = inner] {$ n $};
				\draw[H] (inner.north west) -- (leaf1.south east);
				\draw[H] (inner.north east) -- (leaf2.south west);
			\end{tikzpicture}
		\end{equation*}
		By the constraints on the decorated trees, $T_1$ cannot have $r$ as a decoration at the root. Therefore, one does not have any $r$ decorations in the decorated tree. 
		In the end, we obtain a comb decorated trees (growing only on the left) that encodes the iterated brackets $ \set{H, F_1}^n $. This allows us to conclude on the case $m=1$.
		We proceed with the inductive case and suppose the property true for $ m \in \mathbb{N}^* $. Let us consider
		\begin{equation*}
			F_{m+1} =  \sum_{T \in \, \mathcal{T}^{\! m+2}_{n}} \frac{(\Pi T)}{S(T)}, \quad \set{H_0, F_{m+1}} =  - \sum_{T \in \, \mathcal{T}^{\! m+2}_{\circ}} \frac{(\Pi T)^{\text{\tiny{non-res}}}}{S(T)}.
		\end{equation*}
		Then, we compose $\msH_{m}$ with $F_{m+1}$.
	\begin{equation*}
		\begin{aligned}
			\msH_{m+1}^{\ell} & =(\msH_{m} \circ F_{m+1})^{\ell}  =  H_0 + \sum_{T \in \, \mathcal{T}^{\! < m + 2}_r}  \frac{(\Pi T)}{S(T)} +  \sum_{T \in \, \mathcal{T}^{\! m + 2}_{\circ}}  \frac{(\Pi T)}{S(T)} \\ & + \sum_{T \in \, \mathcal{T}_{\circ}^{m + 2,\ell}}   \frac{(\Pi T)}{S(T)} + \set{H_0, F_{m+1}} + \cdots
		\end{aligned}
	\end{equation*}
where the $\cdots$ contain other terms which are obtained by iterated brackets with $F_{m+1}$. One can notice that
\begin{equation*}
	\begin{aligned}
	\sum_{T \in \, \mathcal{T}^{\! m + 2}_{\circ}}  \frac{(\Pi T)}{S(T)} + \set{H_0, F_{m+1}} & = \sum_{T \in \, \mathcal{T}^{\! m + 2}_{\circ}}  \frac{(\Pi T)}{S(T)}  - \sum_{T \in \, \mathcal{T}^{\! m+2}_{\circ}} \frac{(\Pi T)^{\text{\tiny{non-res}}}}{S(T)}
	\\ &  = \sum_{T \in \, \mathcal{T}^{\! m + 2}_{r}}  \frac{(\Pi T)}{S(T)}. 
	\end{aligned}
\end{equation*}
Then, 
\begin{equation*}
		\begin{aligned}
			\msH_{m+1}^{\ell} & =(\msH_{m+1} \circ F_{m+1})^{\ell}  =  H_0 + \sum_{T \in \, \mathcal{T}^{\! < m + 3}_r}  \frac{(\Pi T)}{S(T)}   + \sum_{T \in \, \mathcal{T}_{\circ}^{m + 2,\ell}}   \frac{(\Pi T)}{S(T)} + \cdots
		\end{aligned}
	\end{equation*}
	Now, one has to get from the $\cdots$ the missing terms. First, one notices the following
	\begin{equation*}
		\mathcal{T}_{\circ}^{m + 2,m+3} 
		\varsubsetneq
			\mathcal{T}_{\circ}^{m + 3}.
	\end{equation*}
Indeed, one is missing the terms that contain subtrees in $ \CT_{\circ}^{m+2} $. These terms are coming from the following brackets:
\begin{equation*}
\sum_{T \in \, \mathcal{T}^{\! m+2}_{n}} 	\set{H_1^{\text{\tiny{res}}},\frac{(\Pi T)}{S(T)}}, \quad \sum_{T \in \, \mathcal{T}^{\! m+2}_{n}}  \set{\set{H_0,H_{1, \Phi}^{\text{\tiny{non-res}}}}, \frac{(\Pi T)}{S(T)}}.
\end{equation*}
One can observe that for $T \in \, \mathcal{T}^{\! m+2}_{n}$,
\begin{equation*}
	\begin{aligned}
	\set{H_1^{\text{\tiny{res}}},\frac{(\Pi T)}{S(T)}} & = \frac{(\Pi \hat{T})}{S(\hat{T})}.
, \quad \hat{T} = 	\begin{tikzpicture}[node distance = 14pt, baseline=10]
		\node[leaf] (inner) at (0,0.3) {$n$};
		\node[leaf] (leaf1) [above left of = inner] {$ r $};
		\node[leaf] (leaf2) [above right of = inner] {$ T $};
		\draw[H] (inner.north west) -- (leaf1.south east);
		\draw[H] (inner.north east) -- (leaf2.south west);
	\end{tikzpicture}.
	\end{aligned}
\end{equation*}
This allows us to get the terms in $ 
\mathcal{T}_{\circ}^{m + 3} \setminus \mathcal{T}_{\circ}^{m + 2,m+3}  $. Now, we need to make sure that we get all the terms in $ 
\mathcal{T}_{\circ}^{m+3,\ell} \setminus \mathcal{T}_{\circ}^{m + 2,\ell}  $. One gets iterated brackets of the form
\begin{equation*}
\frac{1}{p!} \,	\set{ \set{\frac{(\Pi T_1)}{S(T_1)}, \frac{(\Pi T_{2})}{S(T_2)}} \cdots, \frac{(\Pi T_p)}{S(T_p)}  }, 
\end{equation*}
where
\begin{equation*}
	T_1 \in \, \mathcal{T}^{\! < m + 2}_r \, \sqcup \, \mathcal{T}^{\! m + 2}_{\circ} \, \sqcup \, \mathcal{T}_{\circ}^{m + 2,\ell}, \quad T_2,..., T_p  \in \, \mathcal{T}^{\! m+2}_{n}.
\end{equation*}
One has from Proposition\ref{prop_iterated_bracket}
\begin{equation*}
\frac{1}{p!}	\set{ \set{\frac{(\Pi T_1)}{S(T_1)}, \frac{(\Pi T_{2})}{S(T_2)}} \cdots, \frac{(\Pi T_p)}{S(T_p)}  } = \frac{(\Pi T)}{S(T)}, \quad T = \begin{tikzpicture}[node distance = 14pt, baseline=10]
		\node[leaf] (inner) at (0,0.3) {$\circ$};
		\node[leaf] (leaf1) [above left of = inner] {$ \tiny{\cdots} $};
		\node[leaf] (leaf3) [above left of = leaf1] {$ \circ $};
		\node[leaf] (leaf5) [above left of = leaf3] {$ T_1 $};
		\node[leaf] (leaf6) [above right of = leaf3] {$ T_2 $};
		\node[leaf] (leaf2) [above right of = inner] {$ T_p $};
		\draw[H] (inner.north west) -- (leaf1.south east);
		\draw[H] (inner.north east) -- (leaf2.south west);
		\draw[H] (leaf1.north west) -- (leaf3.south east);
		\draw[H] (leaf3.north east) -- (leaf6.south west);
		\draw[H] (leaf3.north west) -- (leaf5.south east);
	\end{tikzpicture}.
\end{equation*}
It is easy to see that $T$ satisfies Assumption~\ref{assumption_1}. Indeed, first one has $|T_2| = \cdots |T_p|$, it remains to check the condition with $T_1$: 
\begin{itemize}
	\item  If $T_1 \in \, \mathcal{T}^{\! < m + 2}_r$, then one has $|T_1|  < |T_2|$ because $T_2 \in \mathcal{T}^{\! m+2}_{n}$.
	\item If $T_1 \in  $  $\mathcal{T}_{\circ}^{m + 2,\ell}$, then $ | T_1| \geq |T_2| $ and   in the decomposition
		\begin{equation*}
		T_1 = \begin{tikzpicture}[node distance = 14pt, baseline=10]
			\node[leaf] (inner) at (0,0.3) {$j $};
			\node[leaf] (leaf1) [above left of = inner] {$ T_4 $};
			\node[leaf] (leaf2) [above right of = inner] {$ T_3 $};
			\draw[H] (inner.north west) -- (leaf1.south east);
			\draw[H] (inner.north east) -- (leaf2.south west);
		\end{tikzpicture}	 
	\end{equation*}
	one has $ |T_3| \leq |T_2| $. This is due to the definition $\mathcal{T}_{\circ}^{m + 2,\ell}$ which guarantees that $|T_3| < m+2$.
		\item If $T_1 \in \mathcal{T}_{\circ}^{m + 2}$, then $ |T_1| \geq |T_2| $ and one has also $  |T_3| \geq |T_2| $.
\end{itemize}
Then, the decorated tree $ T$ belongs to $\mathcal{T}_{\circ}^{m + 3,\ell} \setminus \mathcal{T}_{\circ}^{m + 2,m+3}$ for some $\ell$ sufficiently large. It is easy to see that it is a full description of the sets $\mathcal{T}_{\circ}^{m + 3,\ell} \setminus \mathcal{T}_{\circ}^{m + 2,m+3}$. Indeed, from the definition of $\mathcal{T}_{\circ}^{m + 3,\ell} \setminus \mathcal{T}_{\circ}^{m + 2,m+3}$, one cannot have subtrees of size $ 2m+6 $. One has an order on the size of the subtrees given by Assumption~\ref{assumption_1} which forces us to have decorated trees of the form $T$ where the subtrees of size $2m+4$ are $T_2,...,T_p$. They cannot appear after a subtree of size strictly smaller than $m+2$. This is a consequence of the description given in Definition \ref{def::truncation-term}.
	\end{proof}


\begin{thebibliography}{99}
		\expandafter\ifx\csname url\endcsname\relax
		\def\url#1{\texttt{#1}}\fi
		\expandafter\ifx\csname urlprefix\endcsname\relax\def\urlprefix{URL }\fi
		\expandafter\ifx\csname href\endcsname\relax
		\def\href#1#2{#2}\fi
		\expandafter\ifx\csname burlalt\endcsname\relax
		\def\burlalt#1#2{\href{#2}{\texttt{#1}}}\fi
		
		
		\bibitem{Arnol63}
		V.~I. Arnol'd. 
		\newblock\emph{Proof of a theorem of A. N. Kolmogorov on the preservation of conditionally periodic motions under a small perturbation of the Hamiltonian}.
		\newblock Uspehi Mat. Nauk {\bf 18} (1963), no.~5(113), 13--40.
		\newblock \burlalt{doi:10.1070/RM1963v018n05ABEH004130}{https://doi.org/10.1070/RM1963v018n05ABEH004130}.
		
		\bibitem{Bam03}
		D.~Bambusi.
		\newblock\emph{Birkhoff normal form for some nonlinear PDEs} 
		\newblock Comm. Math. Phys. {\bf 234} (2003), no.~2, 253--285.
		\newblock \burlalt{doi:10.1007/s00220-002-0774-4}{https://doi.org/10.1007/s00220-002-0774-4}.
		
		
		
		
		\bibitem{BeBCE23}
		J.~Bernier, S.~Blanes, F.~Casas, A. Escorihuela-Tom\'as. 
		\newblock\emph{Symmetric-conjugate splitting methods for linear unitary problems}.
		\newblock BIT {\bf 63} (2023), no.~4, Paper No. 58, 26 pp.
		\newblock \burlalt{doi:10.1007/s10543-023-00998-4}{https://doi.org/10.1007/s10543-023-00998-4}.
		
		\bibitem{BeFaG20}
		{\rm J.~Bernier, E.~Faou, B.~Gr\'ebert}.
		\newblock\emph{Long time behavior of the solutions of NLW on the $d$-dimensional torus}.
		\newblock Forum Math. Sigma {\bf 8} (2020), Paper No. e12, 26 pp.
		\newblock \burlalt{doi:10.1017/fms.2020.8}{https://doi.org/10.1017/fms.2020.8}.
		
		\bibitem{BerFG20}
		{\rm J.~Bernier, E.~Faou, B.~Gr\'ebert} 
		\newblock\emph{Rational normal forms and stability of small solutions to nonlinear Schr\"odinger equations}.
		\newblock Ann. PDE {\bf 6} (2020), no.~2, Paper No. 14, 65 pp.
		\newblock\burlalt{doi: 10.1007/s40818-020-00089-5}{https://doi.org/10.1007/s40818-020-00089-5}
		
			\bibitem{BG06}
		D. Bambusi, B.~Grébert. 
		\newblock{\em Birkhoff normal form for partial differential equations with
			tame modulus.} 
		\newblock Duke Math. J.  {\bf 135}, no. 3, (2006), 507--567.
		\newblock \burlalt{doi:10.1215/S0012-7094-06-13534-2}{https://doi.org/10.1215/S0012-7094-06-13534-2}.
		
		\bibitem{BG22}
		J.~Bernier, B.~Grébert. 
		\newblock{\em Birkhoff normal forms for Hamiltonian PDEs in their energy space.} 
		\newblock J. Éc. Polytech. Math.  {\bf 9}, (2022), 681--745.
		\newblock \burlalt{doi:10.5802/jep.193}{https://doi.org/10.5802/jep.193}.
		
		\bibitem{Bou98}
		{\rm J.~Bourgain}. 
		\newblock\emph{Quasi-periodic solutions of Hamiltonian perturbations of 2D linear Schr\"odinger equations}. 
		\newblock Ann. of Math. (2) {\bf 148} (1998), no.~2, 363--439.
		\newblock \burlalt{doi:10.2307/121001}{https://doi.org/10.2307/121001}.
		
		\bibitem{Bou04}
		J.~Bourgain. 
		\newblock{\em A remark on normal forms and the ``$I$-method'' for periodic NLS.} 
		\newblock J. Anal. Math. {\bf 94} (2004), 125--157.
		\newblock \burlalt{doi:10.1007/BF02789044}{https://doi.org/10.1007/BF02789044}.
		
		\bibitem{B24}
		Y.~Bruned.
		\newblock {\textsl{Derivation of normal forms for dispersive PDEs via arborification.}}
		\newblock \burlalt{arXiv:2409.03642}{http://arxiv.org/abs/2409.03642}.
		
		\bibitem{BS}
		Y.~{Bruned}, K.~{Schratz}. \newblock { \em Resonance based schemes for dispersive equations via decorated 
			trees}. Forum of Mathematics, Pi, 10, (2022), E2. 
		\newblock \burlalt{doi:10.1017/fmp.2021.13}{https://doi.org/10.1017/fmp.2021.13}.
		
		\bibitem{CraWa93}
		W.~L. Craig, C.~E. Wayne. 
		\newblock\emph{Newton's method and periodic solutions of nonlinear wave equations}. 
		Comm. Pure Appl. Math. {\bf 46} (1993), no.~11, 1409--1498.
		\newblock\burlalt{doi: 10.1002/cpa.3160461102}{https://doi.org/10.1002/cpa.3160461102}
		
		\bibitem{CKO12}
		{\rm J.~Colliander, S.~Kwon, T.~Oh}.
		\newblock \emph{A remark on normal forms and the “upside-down” I-method for periodic NLS: Growth of higher Sobolev norms}.
		\newblock J. Anal. Math. \textbf{118}, (2012), 55--82.
		\newblock
		\burlalt{doi:10.1007/s11854-012-0029-z}{https://dx.doi.org/10.1007/s11854-012-0029-z}.
		
		\bibitem{CoSYX25}
		{\rm H.~Cong, S.~Li, Y.~Sun, X.~Wu}. 
		\newblock\emph{Birkhoff normal form and long time existence for $d$-dimensional generalized Pochhammer-Chree equation} 
		\newblock J. Stat. Phys. {\bf 192} (2025), no.~2, Paper No. 30, 27 pp.
		\newblock\burlalt{doi:10.1007/s10955-025-03409-w}{https://doi.org/10.1007/s10955-025-03409-w}.
		
		\bibitem{CuEbO20}
		{\rm C. Curry, K. Ebrahimi-Fard, B. Owren}.
		\newblock\emph{The Magnus expansion and post-Lie algebras}. 
		\newblock Math. Comp. {\bf 89} (2020), no.~326, 2785--2799
		\newblock\burlalt{doi: 10.1090/mcom/3541}{https://doi.org/10.1090/mcom/3541}
		
		\bibitem{Grebe06}
		{\rm B.~Gr\'ebert}.
		\newblock\emph{Birkhoff normal form and Hamiltonian PDEs}.
		Partial differential equations and applications, 1--46, S\'emin. Congr., 15, Soc. Math. France, Paris.
		\burlalt{doi: hal-00022311v2}{https://hal.science/hal-00022311}.
		
		\bibitem{GKO13}
		{\rm Z.~Guo, S.~Kwon, T.~Oh}.
		\newblock \emph{Poincaré-Dulac normal form
			reduction for unconditional well-posedness of the periodic cubic NLS}.
		\newblock Comm. Math. Phys. \textbf{322}, no. 1, (2013), 19--48.
		\newblock
		\burlalt{doi:10.1007/s00220-013-1755-5}{https://dx.doi.org/10.1007/s00220-013-1755-5}.
		
		\bibitem{Iserl11}
		{\rm A. Iserles}.
		\newblock\emph{Magnus expansions and beyond}.
		\newblock Combinatorics and physics, 171--186, Contemp. Math., 539, Amer. Math. Soc., Providence, RI.
		\newblock\burlalt{doi: 10.1090/conm/539/10634}{https://doi.org/10.1090/conm/539/10634}
		
		\bibitem{IseNo99}
		{\rm A.~Iserles, S.~P. N\o rsett}.
		\newblock\emph{ On the solution of linear differential equations in Lie groups}.
		\newblock R. Soc. Lond. Philos. Trans. Ser. A Math. Phys. Eng. Sci. {\bf 357} (1999), no.~1754, 983--1019.
		\newblock\burlalt{doi:10.1098/rsta.1999.0362}{https://doi.org/10.1098/rsta.1999.0362}
		
		\bibitem{Kuksi87}
		{\rm S.~B. Kuksin}. 
		\newblock\emph{Hamiltonian perturbations of infinite-dimensional linear systems with imaginary spectrum}.
		\newblock Funct. Anal. Appl. {\bf 21} (1987), no.~3, 22--37, 95.
		\burlalt{doi:10.1007/BF02577134}{https://doi.org/10.1007/BF02577134}.
		
		\bibitem{KP96}
		{\rm S.~B. Kuksin, J.~Pöschel}.
		\newblock \emph{Invariant Cantor manifolds of quasi-periodic oscillations for a nonlinear
			Schrödinger equation}. Ann. of Math.
		\newblock  \textbf{143}, no. 2, (1996), 149--179.
		\newblock
		\burlalt{doi:10.2307/2118656}{https://dx.doi.org/10.2307/2118656}.
			
		\bibitem{LauCo2025}
		{\rm A.B.~Laurent, O. Cosserat}.
		\newblock\emph{Butcher series for Hamiltonian Poisson integrators through symplectic groupoids}.
		\newblock \burlalt{arXiv:2503.05000}{http://arxiv.org/abs/2503.05000}.
			
		\bibitem{Magnu54}
		{\rm W. Magnus}. 
		\newblock\emph{On the exponential solution of differential equations for a linear operator}.
		\newblock Comm. Pure Appl. Math. {\bf 7} (1954), 649--673.
		\newblock\burlalt{doi:10.1002/cpa.3160070404}{https://doi.org/10.1002/cpa.3160070404}
			
		\bibitem{Moser62}
		{\rm J.~K. Moser}.
		\newblock\emph{On invariant curves of area-preserving mappings of an annulus}. 
		\newblock Nachr. Akad. Wiss. G\"ottingen Math.-Phys. Kl. II {\bf 1962} (1962), 1--20.
		
		\bibitem{MunWr08}
		{\rm H.~Z. Munthe-Kaas, W.~M. Wright}. 
		\newblock\emph{On the Hopf algebraic structure of Lie group integrators}.
		\newblock Found. Comput. Math. {\bf 8} (2008), no.~2, 227--257
		\newblock\burlalt{doi: 10.1007/s10208-006-0222-5}{https://doi.org/10.1007/s10208-006-0222-5}
		
		\bibitem{OST18}
		{\rm T.~Oh, P.~Sosoe, N.~Tzvetkov}.
		\newblock \emph{An optimal regularity result on the quasi-invariant Gaussian measures for the cubic
			fourth order nonlinear Schrödinger equation}.
		\newblock  J. Éc. polytech. Math. \textbf{5}, (2018), 793--841.
		\newblock\burlalt{doi: 10.5802/jep.83}{https://doi.org/10.5802/jep.83}
		
		\bibitem{Pere03}
		R.~P\'erez-Marco. 
		\newblock\emph{Convergence or generic divergence of the Birkhoff normal form.} 
		\newblock Ann. of Math. (2) {\bf 157} (2003), no.~2, 557--574.
		\newblock\burlalt{doi: 10.4007/annals.2003.157.557}{https://doi.org/10.4007/annals.2003.157.557}.
		
		\bibitem{Kriko22}
		R.~Krikorian. 
		\newblock\emph{On the divergence of Birkhoff normal forms}
		\newblock Publ. Math. Inst. Hautes \'Etudes Sci. {\bf 135} (2022), 1--181.
		\newblock\burlalt{doi: 10.1007/s10240-022-00130-2}{https://doi.org/10.1007/s10240-022-00130-2}
		
		\bibitem{ST23}
		C.~Sun, N.~Tzvetkov.
		\newblock {\textsl{Quasi-invariance of Gaussian measures for the 3d energy critical nonlinear Schr\"odinger equation.}}
		\newblock \burlalt{arXiv:2308.12758}{http://arxiv.org/abs/2308.12758}.
		\bibitem{Wayne90}
		C.~E. Wayne.
		\newblock\emph{Periodic and quasi-periodic solutions of nonlinear wave equations via KAM theory}.
		\newblock Comm. Math. Phys. {\bf 127} (1990), no.~3, 479--528
		\newblock\burlalt{doi: 10.1007/BF02104499}{http://projecteuclid.org/euclid.cmp/1104180217}
	\end{thebibliography}
\end{document}